%% file: draft.tex
\documentclass[11pt]{article}
\usepackage{amssymb, authblk}
\usepackage{amsmath}
\usepackage{fullpage}
\usepackage{amsthm, natbib, hyperref}
\usepackage{graphicx}
\usepackage{algorithm}
\usepackage{algpseudocode}

\algnewcommand\algorithmicinput{\textbf{INPUT:}}
\algnewcommand\INPUT{\item[\algorithmicinput]}
\algnewcommand\algorithmicoutput{\textbf{OUTPUT:}}
\algnewcommand\OUTPUT{\item[\algorithmicoutput]}

\usepackage{hyperref}[]
\hypersetup{
    colorlinks=true,
    linkcolor=blue,
    filecolor=magenta,      
    urlcolor=cyan,
      citecolor=blue,
    }

\usepackage{cleveref}

\newcommand{\cX}{\{X_i\}_{i=1}^n}

\newtheorem{theorem}{Theorem}
\newtheorem{lemma}[theorem]{Lemma}
\newtheorem{proposition}[theorem]{Proposition}
\newtheorem{corollary}[theorem]{Corollary}

\newtheorem{definition}{Definition}
\newtheorem{remark}{Remark}
\newtheorem{assumption}{Assumption}
\newcommand{\tS}{\widetilde S^{s,e}}
\newcommand{\tSigma}{\widetilde \Sigma^{s,e}}

\newcommand{\tf}{\widetilde f^{s,e}}
\newcommand{\p}{\mathcal P}

\newcommand{\kse}{\kappa^{s,e}_{\max}}

\DeclareMathOperator*{\argmax}{arg\,max}

\DeclareMathOperator*{\op}{op}

\date{\vspace{-5ex}}

\title{Optimal Covariance Change Point Localization in High Dimensions}
\author[1]{Daren Wang}
\author[2]{Yi Yu}
\author[1]{Alessandro Rinaldo}

\affil[1]{\small Department of Statistics and Data Science, Carnegie Mellon University}
\affil[2]{\small School of Mathematics, University of Bristol}

\begin{document}
\maketitle

\begin{abstract}
We study the problem of change point detection for covariance matrices in high dimensions. We assume that we observe a sequence $\{X_i\}_{i=1,\ldots,n}$ of independent and centered $p$-dimensional sub-Gaussian random vectors whose covariance matrices are piecewise constant.  Our task is to recover with high accuracy the number and locations of the change points, which are assumed unknown.  Our generic model setting allows for all the model parameters to change with $n$, including the dimension $p$, the minimal spacing between consecutive change points $\Delta$, the magnitude of smallest change $\kappa$ and the maximal Orlicz-$\psi_2$ norm $B$ of the covariance matrices of the sample points.  Without assuming any additional structural assumption, such as low rank matrices  or having sparse principle components, we set up a general framework and a benchmark result for the covariance change point detection problem.

We introduce two procedures, one based on the binary segmentation algorithm \citep[e.g.][]{vostrikova1981detection} and the other on its extension known as wild binary segmentation of \cite{fryzlewicz2014wild}, and demonstrate that, under suitable conditions, both procedures are able to consistently estimate the number and locations of change points. Our second algorithm, called Wild Binary Segmentation through Independent Projection (WBSIP), is shown to be optimal in the sense of allowing for the minimax scaling in all the relevant parameters. Our minimax analysis reveals a phase transition effect based on the problem of change point localization.  To the best of our knowledge, this type of results has not been established elsewhere in the high-dimensional change point detection literature.
\vskip 5mm
\textbf{Keywords}: Change point detection;  High-dimensional covariance testing; Binary segmentation; Wild binary segmentation; Independent projection; Minimax optimal. 
\end{abstract}

\section{Introduction}
\input{introduction}

\section{Main Results}
\label{section:results}
\input{main}

\section{Lower bounds}
\label{section:lb}
\input{LowerBounds}

\section{Discussion}
\label{section:dis}
\input{Discussion}

\bibliographystyle{ims}
\bibliography{citations}

\appendix
\section*{Appendices}

\input{app}

%% file: introduction.tex

Change point detection in time series data has a long history and can be traced back to World War II, during which the change point detection was specifically demonstrated in a sequential fashion and heavily used in quality control.  \cite{Wald1945} was published shortly after the war and formally introduced sequential probability ratio test (SPRT), which consists of simple null and alternative hypotheses and is based on sequentially calculating the cumulative sum of log-likelihood ratios.  \cite{Page1954} relaxes the statistic to the cumulative sum of weights, which are not restricted to log-likelihood ratios, namely cumulative sum (CUSUM).  Since then, a tremendous amount of efforts have been dedicated to this problem \citep[see, e.g.][]{JamesEtal1987,SiegmundVenkatraman1995} with applications ranging from clinical trials to education testing, signal processing and cyber security.

In the last few decades, with the advancement of technology, an increasing demand from emerging application areas, e.g. finance, genetics, neuroscience, climatology, has fostered the development of statistical theories and methods for change point detection also in an \emph{offline} fashion.  Given a time series data set, instead of assuming it is stationary over the whole time course, it is more robust and realistic to assume that it is stationary only within time segments, and the underlying models may change at certain time points.  Formally, we assume that we observe $n$ variables $X_1, \ldots, X_n$ such that,
	\begin{equation}\label{eq-intro-1}
		X_i \sim \begin{cases}
 		F_1, & i = 1, \ldots, t_1 -1, \\
 		F_2, & i= t_1, \ldots, t_2 -1,\\
 		\ldots, \\ 
 		F_K, & i = t_{K-1}, \ldots, n, \\
 	\end{cases}	
	\end{equation}
	where $F_1, \ldots, F_K$, $K \geq 2$, are distribution functions, $F_{k-1} \neq F_k$, $k =2, \ldots, K$, and $\{t_1, \ldots, t_{K-1}\} \subset \{2, \ldots, n-1\}$ are unknown change point locations.  Testing the existence of change points and estimating the locations of the change points are of primary interest.
	
The simplest and best-studied scenario of model~\eqref{eq-intro-1} is that $\{X_i = f_i + \varepsilon_i\}_{i=1}^n$ is a univariate time series, and $\{f_i\}_{i=1}^n$ is a piecewise constant signal.  Change point detection here means detecting changes in the mean of a univariate time series.  An abundance of literature exists on this model and variants thereof. See \Cref{sec:literature} below for a literature review.

The properties of the covariance are also of key theoretical and practical importance in statistics, and detecting the covariance changes in a time series data set also has a long history. See \Cref{sec:literature} below for some literature review.  In this paper we investigate the problem of change point detection and localization for covariance matrices. We consider what is arguably the most basic setting for the problem in \eqref{eq-intro-1}: we observed a sequence $\{X_1, \ldots , X_n\}$ of independent and centered random vectors in $\mathbb{R}^p$ with respective covariance matrices $\{\Sigma_1, \ldots,\Sigma_n\}$ with that $\max_i \| \Sigma_i \|_{\psi_2} \leq B$.  We assume  that $\Sigma_i = \Sigma_{i+1}$ for all time points $i$ except for a few, which we refer to as the change points. We denote with $\Delta$ the minimal distance between two consecutive change points and with $\kappa$ the magnitude of the smallest change, measured by the operator norm of the difference between the covariance matrices at two consecutive change points. The parameters $p$, $\Delta$, $B$ and $\kappa$ completely characterize the change point localization problem's difficulty, which, for a given $n$, is increasing in $p$ and $B$ and decreasing in $\Delta$ and $\kappa$. In fact, for reason that will become clear later on, it will be convenient to aggregate the parameters $B$ and $\kappa$ into one parameter 
	\[
    \frac{\kappa}{B^2},
	\]
	which effectively plays the role of the signal-to-noise ratio. Our goal is to identify all the change points and to locate them, i.e. to estimate their values accurately.

\subsection{List of contributions}

In the following, we summarize the main contributions of this paper.

\begin{enumerate}

\item We describe and analyze two algorithms for covariance change point localization.   The first one, called BSOP (Binary Segmentation in Operator Norm), is based on an adaptation to the covariance setting of the popular binary segmentation (BS) algorithm \citep[e.g.][]{vostrikova1981detection} for change point detection of the mean of a univariate signal. The BSOP algorithm is rather simple to implement and fast. Under appropriate assumptions, we show in \Cref{prop-bs} that BSOP can consistently estimate all the change points, but with a localization rate that is sub-optimal, especially, when the dimension $p$ is allowed to grow with $n$. This finding is consistent with the corresponding analysis of the BS algorithm for the univariate mean localization problem contained in \cite{fryzlewicz2014wild}, who showed that the BS algorithm is only consistent but possibly sub-optimal.  Our second algorithm, called WBSIP (Wild Binary Segmentation through Independent Projections), is significantly more refined and yields much sharper, in fact minimax rate-optimal, localization rates than BSOP under a set of different and milder assumptions; see \Cref{thm:wbsrp}. The WBSIP procedure combines data splitting and independent projections, which have been shown to be effective for mean change point localization in high-dimensional setting by  \cite{wang2016high}, with the wild binary segmentation (WBS) procedure of \cite{fryzlewicz2014wild}.    We  emphasize that WBSIP is fundamentally different in its mechanics and goals from the method of \cite{wang2016high} and that the theoretical analysis of both the BSOP and the WBSIP procedures, though heavily inspired by \cite{fryzlewicz2014wild}, is significantly more involved  and delivers sharper localization rates. See below for a more detailed comparison with these two aforementioned references.

\item We obtain a lower bound on the localization rate for the problem at hand and demonstrate that WBSIP is in fact minimax rate-optimal, up to a $\log(n)$ factor.    Interestingly, our analysis reveals a phase transition effect over the space of the model parameters: if $\Delta = O \left(B^4 p \kappa^{-2} \right)$ then no consistent estimator of the locations of the change points exists. On the other hand, if $\Delta = \Omega\left(  B^4 p \kappa^{-2} \log(n) \right)$ then WBSIP will guarantee a localization rate of the order $B^4 \kappa^{-2}\log(n)$, which, up to a $\log(n)$ term, is minimax optimal.  
	While consistency of change point estimation for high dimension mean vectors and covariance matrices has been recently tackled by several authors \citep[see, e.g.,][]{BaranowskiEtal2016,wang2016high,AueEtal2009, avanesov2016change}, to the best of our knowledge, this phase transition effect has not been established elsewhere. Overall, our lower bound results and the upper bound on the localization rate afforded by WBSIP procedure provide a complete characterization of the problem of change point localization in the covariance setting described above.


\item In our analysis, we rely on finite sample bounds and allow all the relevant parameters, namely $p$, $\Delta$ and $\frac{\kappa}{B^2}$ to change with $n$. We also do not make any structural assumption on the covariance matrices, such as low-rank or sparsity. Overall, our framework is general, and enables our results to serve as theoretical benchmarks for the study of other covariance change point detection problems.

\end{enumerate}

\subsection{Relevant and related literature}\label{sec:literature}
For the classical problems of change point detection and localization of the mean of a univariate time series, least squares estimation is a natural choice. For example, \cite{YaoAu1989} uses the least squares estimators to show that the distance between the estimated change points and the truths are within $O_p(1)$. \cite{Lavielle1999} and \cite{LavielleMoulines2000} consider penalized least squares estimation and prove consistency in the presence of strong mixing and long-range dependence of the error terms $\varepsilon_i$'s. \cite{HarchaouiLevy2010} consider the least squares criterion with a total variation.

Besides least squares estimation based methods, other attempts have also been made.  For instance, \cite{DavisEtal2006} propose a model based criterion on an autoregressive model. \cite{Wang1995} detects the change points using fast discrete wavelet transform.  \cite{Harchoui}, \cite{Qian_Jia}, \cite{Rojas} and \cite{KevinNIPS} consider the fused lasso \citep{TibshiraniEtal2005} to penalize the differences between the signals at two consecutive points to detect the change points.  \cite{FrickEtal2014} introduce a multiscale statistic based procedure, namely simultaneous multiscale change point estimator, to detect change points in piecewise constant signals with errors from exponential families.  More recently, \cite{LiEtal2017} extends the multiscale method to a class of misspecified signal functions which are not necessarily piecewise constants.  The method of \cite{davies2001} can also be used to determine the locations of change points in a piecewise constant signal. \cite{chan2009} study the power of the likelihood ratio test statistics to detect the presence of change points. 

Among all the methods, BS \citep[e.g.][]{vostrikova1981detection} is `arguably the most widely used change-point search method' \citep{KillickEtal2012}.  It goes through the whole time course and searches for a change point.  If a change point is detected, then the whole time course is split into two, and the same procedure is conducted separately on the data sets before and after the detected change point.  The procedure is carried on until no change point is detected, or the remaining time course consists of too few time points to continue the test.  \cite{venkatraman1992consistency} proves the consistency results of the BS method in the univariate time series mean change point detection, with the number of change points allowed to increase with the number of time points.  

It is worth to mention a variant of BS, namely WBS, which is proposed in \cite{fryzlewicz2014wild} and which can be viewed as a flexible moving window techniques, or a hybrid of moving window and BS.  Instead of starting with the whole data set and doing binary segmentation, WBS randomly draws a collection of intervals under certain conditions, conducts BS on each interval, and return the one which has the most extreme criterion value among all the intervals.  Compared to BS, under certain conditions, WBS is more preferable when multiple change points are present.  In the univariate time series mean change point detection problem, \cite{venkatraman1992consistency} shows that in order to achieve the estimating consistency using BS algorithm, the minimum gap between two consecutive change points should be at least of order $n^{1-\beta}$, where $n$ is the number of time points, and $0 \leq \beta < 1/8$; as claimed in \cite{fryzlewicz2014wild}, by using WBS algorithm, this rate can be reduced to $\log(n)$.

All the literature mentioned above tackles univariate time series models, however, in the big data era, data sets are now routinely more complex and often appear to be multi- or high-dimensional, i.e. $X_i \in \mathbb{R}^p$, where $p$ is allowed to grow with the number of data points $n$.  \cite{HorvathHuskova2012} propose a variant of the CUSUM statistic by summing up the square of the CUSUM statistic in each coordinate.  \cite{ChoFryzlewicz2012} transform a univariate non-stationary time series into multi-scale wavelet regime, and conduct BS at each scale in the wavelet context.  \cite{Jirak2015} allows $p$ to tend to infinity together with $n$, by taking maxima statistics across panels coordinate-wise.  \cite{ChoFryzlewicz2015} propose sparsified binary segmentation method which aggregates the CUSUM statistics across the panel by adding those which exceed a certain threshold.  \cite{Cho2015} proposes the double CUSUM statistics which, at each time point, picks the coordinate which maximizes the CUSUM statistic, and \emph{de facto} transfers the high-dimensional data to a univariate CUSUM statistics sequence.  \cite{AstonKirch2014} introduces the asymptotic concept of high-dimensional efficiency which quantifies the detection power of different statistics in this setting. \cite{wang2016high} study the problem of estimating the location of the change points of a multivariate piecewise-constant vector-valued function under appropriate sparsity assumptions on the number of changes.

As for change point detection in more general scenarios, the SPRT procedure \citep{Wald1945} can be easily used for the variance change point detection.  Based on a generalized likelihood ratio statistic, \cite{BaranowskiEtal2016} tackle a range of univariate time series change point scenarios, including the variance change situations, although theoretical results are missing.  \cite{Picard1985} proposes tests on the existence of change points in terms of spectrum and variance.  \cite{InclanTiao1994} develop an iterative cumulative sums of squares algorithm to detect the variance changes.  \cite{GombayEtal1996} propose some tests on detection of possible changes in the variance of independent observations and obtain the asymptotic properties under the non-change null hypothesis.  \cite{BerkesEtal2009}, among others, extend the tests and corresponding results to linear processes, as well as ARCH and GARCH processes.  \cite{AueEtal2009} considers the problem of variance change point detection in a multivariate time series model, allowing the observations to have $m$-dependent structures.  Note that the consistency results in \cite{AueEtal2009} are in the asymptotic sense that the number of time points diverges and the dimension of the time series remains fixed.  \cite{AueEtal2009} also require the existence of good estimators of the covariance and precision matrices, and the conditions thereof are left implicit.  \cite{BarigozziEtal2016} deal with a factor model, which is potentially of high dimension $p/n = O(\log^2(n))$, and use the wavelet transforms to make the data possibly dependent across the timeline.  Note that the model in \cite{BarigozziEtal2016} can be viewed as a specific covariance change point problem, where the additional structural assumption allows the dimensionality to go beyond the sample size. 

As for the problem of hypothesis testing for high dimensional covariance matrices, which corresponds to the problem of change point detection, the literature is also abundant.  In the cases where $p$ fixed and $n\to \infty$, the likelihood ratio test \citep{Anderson2003} has a $\chi^2_{p(p+1)/2}$ limiting distribution under $H_0: \Sigma = I$.  When both $n, p \to \infty$ and $p/n \to c \in (0, \infty)$, \cite{Johnstone2011} extended Roy's largest root test \citep{Roy1957} and derived the Tracy--Widom limit of its null distribution, to name but a few.  In the case where $n, p \to \infty$ and $p/n \to \infty$, \cite{BirkeDette2005} derived the asymptotic null distribution of the Ledoit--Wolf test \citep{LedoitWolf2002}.  More recently, \cite{cai2013optimal} studied the testing problem in the setting $p \to \infty$ from a minimax point of view and derived the testable region in terms of Frobenius norm is of order $\sqrt{p/n}$.

\vskip 3mm

It is important to highlight the differences between this paper and two closely related papers: \cite{wang2016high} and \cite{fryzlewicz2014wild}.
\begin{itemize}
\item \cite{wang2016high} investigate  a completely different change point detection problem, targeting (sparse) means and not covariances.  Furthermore, even though our main algorithm, based on sample splitting and random projections, is inspired by the methodology proposed in \cite{wang2016high}, it is completely different in its design, properties and goals.  In particular, unlike WBSIP, the algorithm in \cite{wang2016high} relies on semidefinite programming and is indirectly targeting the recovery of the sparse support of a mean vector. The assumptions we make are also rather different: in particular, we do not make any structural assumptions, such as sparsity, nor do we require any eigengap condition.  As a result, the theoretical analysis of our algorithms is also different, and a direct adaptation of their results to our setting will lead to sub-optimal rates.  

\item The WBS algorithm put forward by \cite{fryzlewicz2014wild} is a powerful and flexible methodology for change point localization, and is a key element of our WBSIP procedure, as well as the procedure developed by \cite{wang2016high}. WBS is specifically designed for univariate time series with piece-wise constant mean, whereby the signal strength is decoupled from the noise variance and the optimal solution remains translation invariant.  These features are not present in the more complex covariance setting. Thus, the advantageous properties of the WBS procedure, as demonstrated in \cite{fryzlewicz2014wild}, need not directly apply   to our setting. 
In fact, the accompanying theoretical results about WBS in \cite{fryzlewicz2014wild} are not well suited for our purposes, as they would not directly lead to the optimal rates we manifest for WBSIP (and neither would the alternative analysis of WBS presented \cite{wang2016high}). As a result, we have carried out a more refined and sharper analysis of the performance of WBS  that allows us to derive an optimal dependence on all the model parameters (especially $\kappa$). This improvement is far from trivial and may be of independent interest, as it is also immediately applicable to the mean change point detection problem itself (although we do not pursue this direction here).
\end{itemize}

\subsubsection*{Organization of the paper}

The rest of the paper is organized as follows.  In Section~\ref{section:results}, we propose two different methods for covariance change point detection problem.  All three methods are shown to be consistent, but under different conditions, and the wild binary segmentation with independent projection has the location error rate being $\log(n)$.  In Section~\ref{section:lb}, we show the lower bound of the location error rate in the covariance change point detection problem is $\log(n)$, which implies that the wild binary segmentation with independent projection is minimax optimal.  Further discussion and future work directions can be found in Section~\ref{section:dis}.

\subsubsection*{Notation}
For a vector $v\in \mathbb R^p$ and matrix $\Sigma\in\mathbb R^{p\times p}$, $\|v\|$ and $\|\Sigma\|_{\mathrm{op}} = \max_{\|v\|=1} |v^{\top}\Sigma v|$ indicate the Euclidean and the operator norm, respectively; for any integrable function $f(\cdot): \mathbb{R} \mapsto \mathbb{R}$, denote $\|f\|_1 =\int_{x\in \mathbb{R}} |f(x)|\, dx $ as the $\ell_1$-norm of $f(\cdot)$.  

%% file: main.tex

In this paper, we study the covariance change point detection in high dimension.  To be specific,  we consider a centered and independent time series $\{X_i\} _{i=1}^n \subset \mathbb R^{p}$.  Let $\{\eta_k\}_{k=1}^K \subset \{1,\ldots, n\}$ be the collection of time points at which the covariance matrices of $X_i$'s change.  The model is formally summarized as follows with Orlicz norm $\|\cdot\|_{\psi_2}$, which is defined in \Cref{def-orlicz} in \Cref{section:probability}.

\begin{assumption}\label{assume:model}
	Let $X_1,\ldots, X_n \in \mathbb{R}^p$ be independent sub-Gaussian random vectors such that $\mathbb{E}(X_i)=0$, $\mathbb{E}(X_iX_i^{\top}) = \Sigma_i$ and
	$\|X_i\|_{\psi_2}\le B $ for all $i$. Let  $\{\eta_k\}_{k=0}^{K+1} \subset \{0, \ldots, n\}$ be a collection of change points, such that  $\eta_0=0 $ and $\eta_{K+1}=n$ and that  
		\[
		\Sigma_{\eta_{k} +1} =\Sigma_{\eta_{k} +2} = \ldots  =\Sigma_{\eta_{k+1}},  \text{ for any } k = 0, \ldots, K.
		\]
	Assume the jump size $ \kappa = \kappa(n)$ and the spacing $\Delta = \Delta(n)$ satisfy that
		\[
		\inf_{k = 1, \ldots, K+1} \{\eta_k-\eta_{k-1}\}\ge\Delta > 0,
		\]
		and
		\[
		\|\Sigma_{\eta_k} -\Sigma_{\eta_{k-1} } \|_{\mathrm{op}} = \kappa_{k} \geq \kappa > 0, \text{ for any }  k = 1, \ldots, K+1.
		\]
\end{assumption}

\begin{remark}[{\bf The parameter $B$}]
    The assumption that  $\max_i \|X_i\|_{\psi_2}\le B$ is imposed in order to control the order of magnitude of the random fluctuations of the CUSUM covariance statistics, given below in \Cref{def-1}, around its means. See \Cref{lemma:a1} in the Appendix for details.  At the same time, the chain of inequalities \eqref{eq-kappa-B} shows also that $\max_i \| \Sigma_i \|_{\mathrm{op}} \leq 2 B^2$, so that the
   same assumptions amount also to a uniform upper bound on the operator norm of 
   the covariance matrices of the data points. In this regard,
   $B$ may be reminiscent of the
    assumption that the signal be of bounded magnitude sometimes used in the
    litarature on  mean change point detection: see, e.g.
    \citet[][Assumption~3.1(ii)]{fryzlewicz2014wild} and \citet[][Condition
    (iii) on page 11]{venkatraman1992consistency}. 
    However, while in the mean change point detection problem such boundedness  assumption
    can be in fact removed
     because the optimal solution is translation
    invariant \citep[see, e.g.,][]{wang2016high}, this is not the case in the present
    setting. Indeed, a constant shift in the largest eigenvalue of the population covariance matrices
    also implies a change in the precision with which such matrices can be
    estimated. Thus, it is helpful to think of $B^2$ as some form of variance
    term.
\end{remark}

\begin{remark}[{\bf Relationship between $\kappa$ and $B$}]
    The parameters $\kappa$ and $B$ are not variation independent, as they
    satisfy the inequality $\kappa \leq  B^2/4$. In fact, 
    \begin{equation}
	\kappa \leq \max_{k=1}^K \|\Sigma_{\eta_k} - \Sigma_{\eta_{k-1}}\|_{\mathrm{op}} \leq
 2  \max_{i=1}^n\|\Sigma_i\|_{\mathrm{op}} = 2\max_{i=1}^n \sup_{v \in \mathcal{S}^{p-1}} \mathbb{E}\bigl[(v^{\top}X_i)^2\bigr] \leq 4\max_{i=1}^n\|X_i\|_{\psi_2}^2 \leq 4B^2,
 \label{eq-kappa-B}
\end{equation}
where the second-to-last inequality follows from \Cref{eq:bound.orlictz} in
\Cref{section:probability}. For ease of readability, we will instead use the weaker
bound
$
\kappa \le B^2
$
throughout. In fact, in our analysis we will quantify the combined effect of
both $\kappa$ and $B$
with their ratio $\frac{\kappa}{B^2}$, which we will refer to as the signal-to-noise
ratio. Larger values of such quantities lead to better performance of our
algorithm. It is important to notice that the signal-to-noise ratio, and the
task  of change point detection itself, remains invariant
with respect to any multiplicative rescaling of the data by an arbitrary
non-zero constant. 
\end{remark}

In \Cref{assume:model}, all the relevant parameters $p$, $\Delta$,
$K$, $B$ and $\kappa$ are allowed to be functions of the sample size $n$,
although we do not make this dependence explicit in our notation for ease of
readability. This generic setting allows us to study the covariance change point problem with potentially high-dimensional data, with growing number of change points and decreasing jump sizes.  

Motivated by the univariate CUSUM statistic for mean change point detection, we define the CUSUM statistic in the covariance context.

\begin{definition}[{\bf Covariance CUSUM}]\label{def-1} For $X_1, \ldots, X_n \in
    \mathbb{R}^p$, a pair of integers $(s, e)$ such that $0 \leq s < e-1 < n$,
    and any $t \in \{s+1, \ldots, e-1\}$,  the covariance CUSUM statistic is
    defined as
	\begin{equation*}
		\widetilde S_{t}^{s,e} =\sqrt{\frac{e-t}{(e-s)
		(t-s)}}\sum_{i=s+1}^{t}X_iX_i^{\top}- \sqrt{\frac{t-s}{(e-s)
		(e-t)}} \sum_{i=t+1}^{e} X_iX_i^{\top}. 
	\end{equation*}
	Its expected value is
	\begin{equation*}
		\widetilde \Sigma_{t}^{s,e} =\sqrt{\frac{e-t}{(e-s) (t-s)}}\sum_{i=s+1}^{t}\Sigma_i- \sqrt{\frac{t-s}{(e-s) (e-t)}} \sum_{i=t+1}^{e} \Sigma_i .
	\end{equation*}

\end{definition}

In the rest of this section, we propose two algorithms to detect the covariance change points.  Detailed algorithms and the main consistency results are presented in this section, with the proofs provided in the Appendices.  
The advantages of each algorithm will be discussed later in the section.

\subsection{Consistency of the BSOP algorithm}
\label{sec:bs wbs}
We begin our study by analyzing the performance of a direct adaptation of the binary segmentation algorithm to the matrix setting based on the distance induced by the operator
norm. The resulting algorithm, which we call BSOP, is given in \Cref{alg:BSOP}.
The BSOP procedure works as follows: given any time
interval $(s, e)$, BSOP first computes the maximal operator norm of the
covariance CUSUM statistics over the time points in $(s + \lceil p \log (n)
\rceil,e - \lfloor p \log (n) \rfloor )$; if such maximal
value exceeds a predetermined threshold
$\tau$, then BSOP will identify the location $b$ of the maximum as a change point. The
interval $(s, e)$ is then split into two subintervals at $b$ and the procedure
is then iterated separately on each of the resulting  subintervals $(s, b)$ and
$(b, e)$ until an appropriate stopping condition is met. 

The BSOP algorithm differs
from the standard BS implementation in one aspect: the maximization of
the of the operator
norm of the CUSUM covariance operator is carried out only over the time points in $(s,e)$
that are away by at least $p \log (n)$ from the
endpoints of the interval. 
Such
restriction is needed to obtain adequate tail bounds for the operator norm
of the covariance
CUSUM statistics $\widetilde{S}_{t}^{s,e}$ given in \Cref{def-1} and of the centered
and weighted empirical covariance matrices.  See
\Cref{lemma:a1} in \Cref{section:probability}.

\begin{algorithm}[htbp]
\begin{algorithmic}
	\INPUT $\{X_i \}_{i=s+1}^e \subset \mathbb{R}^{p\otimes (e-s)}$, $\tau > 0$.
	\State {\bf Initial} $\mathrm{FLAG} \leftarrow 0$,
	\While{$e-s > 2p\log(n) + 1$ and $\mathrm{FLAG} = 0$}
		\State $a \leftarrow \max_{\lceil s+ p\log(n) \rceil \leq t \leq \lfloor e-p\log(n) \rfloor} \bigl\| \tS_t \bigr\|_{\mathrm{op}}$
		\If{ $a \leq \tau$}
			\State $\mathrm{FLAG} \leftarrow 1$
			\Else
				\State $b \leftarrow \arg\max_{\lceil s+ p\log(n) \rceil \leq t \leq \lfloor e-p\log(n) \rfloor}  \bigl\| \tS_t \bigr\|_{\mathrm{op}}$
				\State add $b$ to the collection of estimated change points
				\State $\mathrm{BSOP} ((s,b-1),\tau)$
				\State $\mathrm{BSOP} ((b,e),\tau)$
		\EndIf
	\EndWhile
	\OUTPUT The collection of estimated change points.
\caption{Binary Segmentation through Operator Norm. $\mathrm{BSOP}((s,e), \tau)$ }
	\label{alg:BSOP}
\end{algorithmic}
\end{algorithm}

To analyze the performance of the BSOP algorithm we will impose the following
assumption, which is, for the most part, modeled after Assumption 3.2 in
\cite{fryzlewicz2014wild}.  

\begin{assumption}\label{ass-delta-bs}
For a sufficiently large constant $C_\alpha > 0$ and sufficient small constant $c_\alpha > 0$, assume that $\Delta\kappa B^{-2} \ge C_\alpha n^{\Theta}$, $ p \le c_\alpha n^{8\Theta -7}/\log(n)$, where $\Theta \in(7/8,1]$.
\end{assumption}

When the parameters $\kappa$ and
$B$ are fixed, the above assumption requires $\Delta$, the minimal spacing
between consecutive change point, to be of at least slightly smaller
order than $n$, the size of the time series. This is precisely Assumption 3.3
in 
\cite{fryzlewicz2014wild}  \citep[see also][]{ChoFryzlewicz2015}.  The fact that
$\Delta$ cannot be too small compared to $n$ in order for the BS algorithm to
exhibit good performance is well known:  see, e.g., \cite{olshen2004circular}.  In
\Cref{ass-delta-bs}, we require also the dimension $p$ to be upper bounded by
$\frac{n^{8\Theta -7}}{\log (n)}$, which means that $p$ is allowed to diverge as $n \to \infty$.

\begin{remark}[{\bf Generalizing \Cref{ass-delta-bs}}] 
In \Cref{ass-delta-bs} we impose certain constraints on the scaling of the quantities $B$,
$\kappa$, $\Delta$ and $p$ in relation to $n$ that are captured by a single parameter
$\Theta$, whose admissible values lie in $(7/8,1]$. The strict lower bound of
$7/8$ on
$\Theta$ is determined by the calculations outlined below on
\eqref{eq:theorem1 first basic} and \eqref{eq:theorem1 second basic}, which are needed to
ensure the existence of a non-empty range of value for the input parameter $\tau$
to the BSOP algorithm. 
In fact, \Cref{ass-delta-bs} may be generalized 
by allowing for different types of scaling in $n$ of the signal-to-noise ratio $\kappa B^{-2}$,
 the minimal distance $\Delta$ between consecutive change points and
the dimension $p$. In detail,  we may require that
$\kappa B^{-2} \succeq n^{\Theta_1}$, $ \Delta \succeq n^{\Theta_2}$ and $p \log (n) \preceq n^{\Theta_3}$ for a given triplet of parameters
$(\Theta_1,\Theta_2,\Theta_3)$ in an appropriate subset of
$[0,1]^{\otimes 3}$.  Such a generalization would then lead to consistency rates in $n$ that
depend on all these parameters simultaneously. However, the range of allowable values of $(\Theta_1, \Theta_2, \Theta_3)$ is
not a product set due to non-trivial constraints among them.
We will refrain from providing details and instead rely on the simpler
formulation given in \Cref{ass-delta-bs}. 
\end{remark}

\begin{theorem}[Consistency of $\mathrm{BSOP}$]\label{prop-bs}
Under Assumptions~\ref{assume:model} and \ref{ass-delta-bs}, let $\mathcal{B} =
\{\hat \eta_k\}_{k=1}^{\widehat K}$ be the collection of the estimated change
points from the $\mathrm{BSOP}((0, n), \tau)$ algorithm, where the parameter
$\tau$ satisfies
	\begin{equation}\label{eq:tau BSOP}
	B^2\sqrt{p\log(n)} + 2\sqrt {\epsilon_n } B^2< \tau < C_1\kappa\Delta n^{-1/2},
	\end{equation}    
	for some constant $C_1 \in(0, 1)$ and where
	\[
\epsilon_n = C_2B^2\kappa^{-1}n^{5/2} \Delta^{-2}\sqrt {p \log(n)}
	\]
	for some $C_2>0$.
	Then,
	\begin{equation}\label{eq-thm1-result}
	    \mathbb{P}\Bigl(   \widehat K = K  \quad \text{and} \quad  \max_{k=1,\ldots,K}
	|\eta_k-\hat \eta_k| \le \epsilon_n  \Bigr) \geq 1- 2 \times 9^p n^3   n^ {-cp},
	\end{equation}
	for some absolute constant $c>0$.
\end{theorem}
\begin{remark} The condition 
    \eqref{eq:tau BSOP} on the admissible values of the input parameter $\tau$ to the $\mathrm{BSOP}$ algorithm is 
   well defined.  
    Indeed, by \Cref{ass-delta-bs}, for all pairs $(s,e)$ such that $e - s
    > 2 p \log (n)$, we have that 
\begin{align}
B^2 \sqrt {p \log (n) } \le B^2 c_\alpha n^{4\Theta - 7/2 }\le B^2 c_\alpha n^{\Theta} n^{-1/2}\le \frac{c_\alpha}{C_\alpha} \kappa \Delta n^{-1/2}
\le (1/8)  \kappa \Delta (e-s)^{-1/2}
\label{eq:theorem1 first basic}
\end{align}
and
\begin{align}
2\sqrt {\epsilon_n} B^2  &= 2 C_2^{1/2} B^3 \kappa^{-1/2} n^{5/4}\Delta^{-1} (p \log (n))^{1/4} \le (2 C_2^{1/2}  C_\alpha^{-1}c_\alpha^{1/4} )B \kappa^{1/2} n^{5/4 +\Theta -7/4} \nonumber\\
&\le (C_1 /8) \kappa \Delta n^{-1/2}B^{-1}\kappa^{1/2} \le  (C_1/8)  \kappa \Delta (e-s)^{-1/2}\label{eq:theorem1 second basic},
\end{align}
where in the chain of inequalities we have used \Cref{ass-delta-bs} repeatedly. It is also worth noting that
the difference between the right-hand-side and the left-hand-side of
\eqref{eq:tau BSOP} increases as $\Theta$ increases to 1. Finally we remark that in the proof of \Cref{prop-bs}, we actually let $C_1 =
    1/8$, but this is an arbitrary choice and it essentially depends on the
    constants
    $C_{\alpha}$ and $c_{\alpha}$ from \Cref{ass-delta-bs}.

\end{remark}

\begin{proof}[Proof of \Cref{prop-bs}]
  By induction, it suffices to consider any pair of
integers $s$ and $e$ such that $(s, e) \subset (0, T)$ and satisfying
	\begin{align*}
		\eta_{r-1} \le s\le \eta_r \le \ldots\le \eta_{r+q} \le e \le \eta_{r+q+1}, \quad q\ge -1, \\
		\max \{ \min \{ \eta_r-s ,s-\eta_{r-1}  \}, \min \{ \eta_{r+q+1}-e, e-\eta_{r+q}\} \} \le \epsilon_n, 
	\end{align*}
	where $q = -1$ indicates that there is no change point contained in $(s, e)$.
It follows  that, for sufficiently small $ c_\alpha > 0$ and
sufficiently large $C_\alpha > 0$, 
	\begin{align*}
	\frac{\epsilon_n}{ \Delta /4} 
	&
	\le \frac{C_2 B^2\kappa^{-1}n^{5/2}\sqrt {p \log(n)} \Delta^{-2} }{\Delta/4 }\\
	&\le 4 C_2 B^2\kappa^{-1}n^{5/2} \frac{ c_\alpha ^{1/2}n^{4\Theta -7/2} }{C_\alpha^3 \kappa^{-3}B^{6} n^{3\Theta}}
	\\
	&\le  ( 4C_2 c_\alpha ^{1/2}C_\alpha^{-3}  ) (\kappa^2 B^{-4}) n^{\Theta
	-1} \\
	& \le (\kappa^2 B^{-4}) n^{\Theta -1} \\ 
	&\le  1
	\end{align*}
	where the second inequality stems from \Cref{ass-delta-bs}, the third
	inequality holds by choosing  sufficiently small $ c_\alpha$ and
sufficiently large $C_\alpha$ and the last inequality follows from the fact that $ \kappa\le B^2$.
Then, for any change point should be $(s,e)$, it is either the case that 
	\[
	|\eta_p - s|\le \epsilon_n,
	\] 
	or that 
	\[
	|\eta_p - s| \ge \Delta - \epsilon_n \geq \Delta - \Delta/4 = 3\Delta/4.
	\] 
	Similar considerations apply to the other endpoint $e$.
	As a consequence, the fact that $ \min\{ |\eta_p-e|, |\eta_p-s| \}\le
	\epsilon_n$ implies that $\eta_p$ is a detected change point found in
	the previous induction step, while if $ \min\{ \eta_p -s, \eta_p-e\} \ge
	3\Delta/4 $ we can conclude that $\eta_p \in (s, e)$ is an undetected
	change point.

In order to complete the induction step, it suffices to show that BSOP($(s, e),\tau$)  
	(i) will not find any new change point in the interval $(s,e)$ if
	    there is
	    none, or if all the change points in $(s, e)$ have been already
	    detected and 
	(ii) will identify a location $b$ such that $|\eta_p - b| \le \epsilon_n$ if there exists at least one undetected change point in $(s, e)$.

Set $\lambda = B^2\sqrt{p\log(n)}$. Then, the event $\mathcal
A_1(\{X_i\}_{i=1}^n ,\lambda)$ defined in \Cref{eq-event-A1} holds with
probability at least $1 -  2 \times 9^p n^3 n^{−cp}$, for some universal
constant $c>0$.  The proof will be completed in two steps.
\vskip 3mm
\noindent {\bf Step 1.} First we will show that on the event $\mathcal A_1(\{X_i\}_{i=1}^n, \lambda)$, BSOP$((s, e), \tau)$ can consistently detect or reject the existence of undetected change points within $(s, e)$.

Suppose there exists $\eta_p \in (s, e)$ such that $ \min\{ \eta_p -s,
\eta_p-e\} \ge 3\Delta/4 $. Set $\delta = p \log (n)$. Then $\delta \leq
\frac{3}{32} \Delta$, since  
	\[
 p\log(n)\le c_\alpha n^{8\Theta- 7}\le c_\alpha n^{\Theta}\le  c_\alpha C_\alpha^{-1} \Delta B^{-2} \kappa^{1}\le  3\Delta/32,
	\]
	where the last inequality follows from \Cref{ass-delta-bs}. With this
	choice of $\delta$, we apply \Cref{lemma:lower bound of CUMSUM} in
	\Cref{sec-pre-1} (where
	we set $c_1 = 3/4$) and
	obtain that
		\[
		\max_{t = \lceil s+\delta \rceil, \ldots, \lfloor e-\delta \rfloor} \| \widetilde \Sigma_{t}^{s,e}\|_{\mathrm{op}}  \ge (3/8)\kappa\Delta (e-s)^{-1/2}.
		\]
		
On the event $\mathcal A_1(\{X_i\}_{i=1}^n ,\lambda)$, 
\begin{align}\label{eq:bs proof size of population}
	\max_{t = \lceil s+\delta \rceil, \ldots, \lfloor e-\delta \rfloor} \| \widetilde S_{t}^{s,e}\|_{\mathrm{op}}& \ge 	\max_{t = \lceil s+\delta \rceil, \ldots, \lfloor e-\delta \rfloor} \| \widetilde \Sigma_{t}^{s,e}\|_{\mathrm{op}} -\lambda \ge  (3/8)\kappa\Delta (e-s)^{-1/2} -\lambda
	 \ge (1/8)\kappa\Delta (e-s)^{-1/2},
	\end{align}
	where the last inequality follows from \eqref{eq:theorem1 first basic}
	(in the last step we have set $C_1 = 1/8$).
	If \eqref{eq:tau BSOP} holds, then, on the event $\mathcal
	A_1(\{X_i\}_{i=1}^n ,\lambda)$, BSOP($(s, e), \tau$) detects the existence of undetected change points if there are any.

Next, suppose there does not exist any undetected change point within $(s, e)$.
Then, one of the following cases must occur. 
\begin{itemize}
		\item [(a)]	There is no change point within $(s, e)$;
		\item [(b)] there exists only one change point $\eta_{r}$ within $(s, e)$ and $\min \{ \eta_{r}-s , e-\eta_{r}\}\le \epsilon_n $;  
		\item [(c)] there exist two change points $\eta_{r} ,\eta_{r+1}$ within $(s, e)$ and that  $\max \{ \eta_{r}-s , e-\eta_{r+1}\}\le \epsilon_n $.
	\end{itemize}

Observe that if case (a) holds, then on the event $\mathcal A_1(\{X_i\}_{i=1}^n
,\lambda)$, we have that
	\[
	\max_{t = \lceil s+\delta \rceil, \ldots, \lfloor e-\delta \rfloor} \| \widetilde S_{t}^{s,e}\|_{\mathrm{op}} \le \max_{t = \lceil s+\delta \rceil, \ldots, \lfloor e-\delta \rfloor} \| \widetilde \Sigma_{t}^{s,e}\|_{\mathrm{op}} +\lambda = \lambda < \tau,
	\]
	where the last inequality follows from \eqref{eq:tau BSOP}.
If situation (c) holds, then on the event  $\mathcal A_1(\{X_i\}_{i=1}^n ,\lambda)$, we have
	\[
	\max_{t = \lceil s+\delta \rceil, \ldots, \lfloor e-\delta \rfloor} \| \widetilde S_{t}^{s,e}\|_{\mathrm{op}} \le \max_{t = \lceil s+\delta \rceil, \ldots, \lfloor e-\delta \rfloor} \| \widetilde \Sigma_{t}^{s,e}\|_{\mathrm{op}} +\lambda\le \max \{\| \widetilde \Sigma_{\eta_r}^{s,e}\|_{\mathrm{op}} ,\| \widetilde \Sigma_{\eta_{r+1}}^{s,e}\|_{\mathrm{op}} \} +\lambda \le 2\sqrt {\epsilon_n} B ^2+\lambda,
	\]
	where the first inequality follows from $\mathcal A_1(\{X_i\}_{i=1}^n
	,\lambda)$, the second inequality from \Cref{lemma:maximized at change
	points} and the third inequality from  \Cref{lemma:size of boundary}.
	(Both Lemmas are in \Cref{sec:properties.cusum}.)
Case (b) can be handled in a similar manner. Thus, if \eqref{eq:tau BSOP} holds,
 then on the event
$\mathcal A_1(\{X_i\}_{i=1}^n ,\lambda)$, BSOP($(s, e),\tau$) has no false
positives when there are no undetected change points in $(s, e)$.
\vskip 3mm

\noindent {\bf Step 2.} Assume now that there exists a change point $\eta_p\in
(s, e)$ such that $ \min\{ \eta_p -s, \eta_p-e\} \ge 3\Delta/4 $ and 
	let
	\[
	b \in \argmax_{t = \lceil s+\delta \rceil, \ldots, \lfloor e-\delta \rfloor}  \bigl\| \tS_t \bigr\|_{\mathrm{op}}.
	\]
	To complete the proof it suffices to show that $|b-\eta_k|\le \epsilon_n$.

Let $v$ be such that
	\[
	v\in \argmax_{\|u\|=1} |u^{\top} \tS_{b} u|.
	\]
	Consider the univariate time series  $\{Y_i (v)\}_{i=1}^n$ and $\{f_i
	(v)\}_{i=1}^n$ defined in \eqref{eq:project sample} and
	\eqref{eq:project population} in \Cref{sec:properties.cusum}.  By
	\Cref{lemma:reduction to 1d}, $b \in \arg\max_{s\le t\le e}
	|\widetilde Y_t(v) |$. Next, we wish to apply \Cref{coro:1d bs} to the time
	series $\{Y_{i} (v)\}_{i=s}^e$ and $\{f_{i} (v)\}_{i=s}^e$. Towards that
	end, we first need to ensure that the conditions required for that
	result to hold are verified.  
(Notice that in  the statement of \Cref{coro:1d bs}, the  $f_i$'s are assumed to
be uniformly bounded by $B_1$, while in this proof the 
$f_{i} (v) $'s defined in  \eqref{eq:project population} are assumed to be bounded by $2B^2$.)
First, the collection of the change points of the time series $\{f_{i}
(v)\}_{i=s+1}^e$ is a subset of $\{\eta_{k}\}_{k=0}^{K+1}\cap (s,e)$.
The condition \eqref{prop:1d bs 4} and the inequality $2\sqrt {\delta} B^2\le  (3c_1/4)\kappa \Delta (e-s)^{-1/2}$ are  straightforward consequences of
\Cref{ass-delta-bs}, while \eqref{coro:1d bs 1} follows from the fact that
	\[
	|\widetilde f_t^{s,e} (v)  -\widetilde Y_t ^{s,e} (v) | \le  \left\|
	\tS_t  -\tSigma_t \right\|_{\op} \le \lambda.
	\]
	Similarly,  \eqref{coro:1d bs 2} stems  from the relations 
	\[
	\max_{t = \lceil s+\delta \rceil, \ldots, \lfloor e-\delta \rfloor} |\widetilde Y_t^{s,e} (v)  |= \max_{t = \lceil s+\delta \rceil, \ldots, \lfloor e-\delta \rfloor}   \|\tS_t \|_{\op}  \ge\max_{t = \lceil s+\delta \rceil, \ldots, \lfloor e-\delta \rfloor} \|\tSigma_t \|_{op}-\lambda \ge (1/8)\kappa \Delta (e-s)^{-1/2}  
	 \]
	where the first inequality holds on the event $\mathcal
	A_1(\{X_i\}_{i=1}^n ,\lambda)$ and the second inequality is due to
	\eqref{eq:bs proof size of population} and \Cref{ass-delta-bs}. Thus,
	all the assumptions of \Cref{coro:1d bs} are met. An application of that
	result yields that there exists $\eta_{k}$, a change point of
	$\{f_i(v)\}_{i=s}^e $ satisfying \eqref{eq:1d re1}, such that 
	\[
	| b-\eta_{k}|\le C_2\lambda(e-s)^{5/2} \Delta^{-2} \kappa^{-1}\le \epsilon_n.
	\]
	
The proof is complete by observing that \eqref{eq:1d re1}  implies $ \min\{ \eta_k -s, \eta_k-e\} \ge 3\Delta/4 $, as discussed in the argument before {\bf Step 1}.

\end{proof}

The proof of the theorem relies on a non-trivial extension of the arguments for proving
consistency of the BS algorithm in one-dimensional mean change point detection
problems, as done in \cite{venkatraman1992consistency}. The main difficulty that prevents a direct
application of those results
is the fact the regions of monotonicity of the function $t \mapsto
\bigl\|\tSigma_t\bigr\|_{\mathrm{op}}$ are hard to derive. Instead, for each pair
of integers $1 \leq s < e \leq n$ with $e - s > 2 p \log (n)$, we study the
one-dimensional time series $\{ (v^\top X_i)^2 \}_{i=1,\ldots,n}$ of the squared
coefficients of the projection of the data along a one-dimensional linear
subspace spanned by a distinguished unit vector $v$, which we term the {\it shadow vector}.
This is simply the leading singular vector of
$\tSigma_b$, where $b = \argmax_{t \in
(s + p \log (n) , e - p \log (n))}\bigl\|\tSigma_t\bigr\|_{\mathrm{op}} $.
As it turns out, with such a choice of the shadow vector, the local maxima of
CUSUM statistic applied to the corresponding one-dimensional time series
coincide with the local maxima of the time series of the values of the operator norm of the
CUSUM covariance statistics. As a result, for the purpose of detecting local maxima of
the CUSUM covariance statistic, it is enough and in fact much simpler to study the univariate times series of the squared
projections onto the appropriate shadow vector. Of course, at each iteration of the BSOP algorithm a new shadow
vector and a new univariate time series are obtained and a new local maximum is
found.  Note that the shadow vector we use here does not need to possess consistency, therefore our algorithm is computationally efficient and attain both tight localization bound and tight phase transition bound.  See \Cref{sec:properties.cusum} for further comments
on the uses and interpretation of the shadow vector.

\Cref{prop-bs} implies, that with high probability, the BSOP algorithm will
identify all the change points and estimate their locations with an error that
is bounded by
\[
  \epsilon_n \preceq  \frac{B^2}{\kappa}  \Delta^{-2} n^{5/2} \sqrt{p \log (n)}.
\]
Notice that, as expected, the performance of BSOP is deteriorating in the inverse
of the
signal-to-noise ratio parameter $\frac{\kappa}{B^2}$, the inverse of the minima
distance $\Delta$ between change points and the dimension $p$.  
The above bound yields a family of rates of consistency for BSOP, depending on
the scaling of each of the quantities involved in it. 
For example, in the simplest and most favorable scenario whereby $B$, $\kappa$ and the dimension $p$ are
constants, the bound 
implies a rate for change point localization of the order 
\[
    \frac{\epsilon_n}{n} \preceq  n^{- 2
\Theta + 3/2} \sqrt{\log (n)},
\] 
which is decreasing in the $\Theta \in (7/8,1]$. In particular, when the number of change points is also kept constant, we have that $\Theta =
1$, yielding a localization rate of order $\sqrt{\frac{ \log (n)}{n}}$. 

As we will see in the next section, the dependence on the parameters $B$, $\kappa$ and $\Delta$ is sub-optimal. 
The advantage of BSOP  over the rate-optimal algorithm we introduce next is that
BSOP only requires one input parameter, the threshold value $\tau$. Furthermore, 
when the spacing parameter $\Delta$ is comparable with $n$ and the dimension $p$
of the data grows slowly with respect with $n$, then BSOP can still deliver good
consistency rates.  Therefore, despite its suboptimality in general, BSOP is a
simple and convenient algorithm which may serve as a competitive benchmark for other
procedures.

\subsection{Consistency of the WBSIP algorithm} 

In this section we describe and analyze the performance of a new algorithm for
covariance change-point detection, which we term WBSIP for Wild Binary Segmentation
through Independent Projections. The WBSIP algorithm is a generalization of the
WBS procedure of \cite{fryzlewicz2014wild} for mean 
change point detection and further exploits the properties of shadow vectors. The WBSIP procedure
begins by splitting the data into halves and by selecting at random a collection
of $M$ pairs of 
integers $(s,e)$ such that $1 \leq s < e \leq n$ and $e - s > p \log (n) + 1$. In
its second step, WBSIP computes, for each of the $M$ random integer intervals
previously generated, a shadow vector using one half of the data and its
corresponding one-dimensional time series using the other half. The final step
of the procedure is to apply
the WBS algorithm over the resulting univariate time series. The details of the
algorithm are given in \Cref{algorithm:PC}, which describes the computation of the
shadow vectors by principal component methods, and \Cref{algorithm:WBSRP}, which
implements WBS to the resulting one
dimensional time series. 

\begin{remark}
The idea of combining the WBS algorithm with sample splitting is previously used in \cite{wang2016high}, who applied it to the problem of mean change-point detection in multivariate settings.  Due to the fact that they are recovering a sparse leading eigenvector in a possibly ultrahigh-dimensional setting, their method is inevitably computationally more expensive, and hard to achieve a tight bounds in terms of the sparsity level.  This leads to one of the main differences between our approach and theirs -- we do no require the shadow vectors to be consistent estimators of subspaces related  the true covariance matrices. In particular, our analysis holds without any eigengap assumptions.
\end{remark}

In order to analyze the performance of the WBSIP procedure, we will impose the
following assumption, which is significantly weaker than \Cref{ass-delta-bs}.

\begin{assumption}\label{assume:phase}
There exists a sufficiently large absolute constant $C >0$ such that $\Delta \kappa^2\ge C p\log(n) B^{4} $.
\end{assumption}

\begin{remark}
We recall that all the  parameters $\Delta, \kappa,p, B$ are allowed to depend
on $n$.  Since $\kappa\le B^2$, and assuming without loss of generality that
the constant $C$ in the previous assumption is larger than $8$, we further have that
	\[
	p\log(n) \le \Delta \kappa ^2 B^{-4} C^{-1} \le \Delta /8,
	\]
	which is used repeatedly below. In fact, in the proof we will set
	 $C = 32\sqrt{2}$. This choice 
	is of course arbitrary and is only made for convenience in carrying out
	the calculations below. 
We also recall that the ratio $B^4\kappa^{-2}$ is invariant of multiplicative
scaling, i.e. if $ X'_i = \alpha X_i$ for all $i$, then the corresponding ratio $B^4\kappa^{-2}$ stays the same.
\end{remark}

\begin{algorithm}[!ht]
\begin{algorithmic}
	\INPUT $\{X_i\}_{i=1}^n$, $\{ (\alpha_m ,\beta_m) \}_{m=1}^M$
	\For{$m = 1, \ldots, M$} 
	\If {$\beta_m-\alpha_m> 2p\log(n) +1 $}
		\State $d_m \leftarrow \argmax_{\lceil \alpha_m +  p\log(n) \rceil \leq t \leq \lfloor \beta_m -  p\log(n) \rfloor}
		 \|\widetilde{S}^{\alpha_m, \beta_m}_t \|_{\mathrm{op}}$
		\State $u_m \leftarrow \argmax_{\|v\| =1} \bigl| v^{\top}\widetilde{S}^{\alpha_m, \beta_m}_{d_m}v|$
		\Else 	\State  $u_m\leftarrow0$
		\EndIf
	
	\EndFor 
	\OUTPUT $\{ u_m\}_{m=1}^M$.
\caption{Principal Component Estimation $\mathrm{PC}(\{X_i\}_{i=1}^n, \{(\alpha_m ,\beta_m)\}_{m=1}^M)$}
\label{algorithm:PC}
\end{algorithmic}
\end{algorithm}

\begin{algorithm}[!ht]
\begin{algorithmic}
	\INPUT Two independent samples $\{W_i\}_{i=1}^{n}$, $\{X_i\}_{i=1}^{n}$, $\tau$, $\delta$.
	\State  $\{ u_m\}_{m=1}^M= PC(\{W_i\}_{i=1}^n, \{(\alpha_m ,\beta_m)\}_{m=1}^M  )$
	\For{$i \in  \{s, \ldots, e\}$}
		\For{$m = 1, \ldots, M$}
			\State $Y_i(u_m) \leftarrow \bigl(u_m^{\top} X_i\bigr)^2$
		\EndFor
	\EndFor 
	\For{$m = 1, \ldots, M$}  
		\State $(s_m', e_m') \leftarrow [s,e]\cap [\alpha_m,\beta_m]$
				and  $(s_m,e_m) \leftarrow (\lceil  s_m'+\delta \rceil , \lfloor e_m'-\delta \rfloor  ) $
		\If{$e_m - s_m \geq 2\log(n) + 1$}
			\State $b_{m} \leftarrow \argmax_{s_m +\log(n) \leq t \leq e_m -\log(n)}  | \widetilde Y^{s_m,e_m}_{t} (u_{m})|$
			\State $a_m \leftarrow \bigl| \widetilde Y^{s_m,e_m}_{b_{m}}  ( u_{m})\bigr|$
		\Else 
			\State $a_m \leftarrow -1$	
		\EndIf
	\EndFor
	\State $m^* \leftarrow \argmax_{m = 1, \ldots, M} a_{m}$
	\If{$a_{m^*} > \tau$}
		\State add $b_{m^*}$ to the set of estimated change points
		\State WBSIP$((s, b_{m*}),\{ (\alpha_m,\beta_m)\}_{m=1}^M, \tau, \delta)$
		\State WBSIP $((b_{m*}+1,e),\{ (\alpha_m,\beta_m)\}_{m=1}^M,\tau, \delta) $

	\EndIf  
	\OUTPUT The set of estimated change points.
\caption{Wild Binary Segmentation through Independent Projection. WBSIP$((s, e),$ $\{ (\alpha_m,\beta_m)\}_{m=1}^M, \tau, \delta$)}
\label{algorithm:WBSRP}
\end{algorithmic}
\end{algorithm}

    Similarly to the BSOP algorithm,  WBSIP also applies a slight modification to
    the WBS algorithm as originally proposed in \cite{fryzlewicz2014wild}. When
    computing the shadow vectors in \Cref{algorithm:PC},  the search for the
    optimal direction onto which to project the data is
    restricted, for any given candidate interval, only
    to the time points that are at least $p \log (n)$ away from the endpoints of
    the interval. As remarked in the previous section, this ensures good tail
    bounds for the operator norms of the matrices involved. 

    A second, more substantial, adaptation of WBS is used in
    \Cref{algorithm:WBSRP}: when searching for candidate change points
    inside a given interval, the algorithm only considers time points that are
    $\delta$-away from the endpoints of the interval, where $\delta$ is an upper
    bound on the localization error -- the term $\epsilon_n$ in \Cref{thm:wbsrp} below. 
The reasons for such restriction are somewhat subtle: once an estimated change point is
found near a true change point, in its next iteration the  algorithm can 
no longer look for change points in that proximity, since this will
result,
with high probability, in spurious detections. This phenomenon is due to the
fact that the behavior of the CUSUM statistic of the projected data is not
uniform around its local maxima. Therefore,  after a true
change point has been detected,  the algorithm must scan
only nearby regions of low signal-to-noise ratio -- so that the probability of false positives can be
adequately controlled with a proper choice of the thresholding parameter
$\tau$. Thus the need to stay away by at leat the localization error from the
detected change points.
A very similar condition is  imposed
in the main algorithm of
\cite{wang2016high} and implicitly in \cite{korkas2017multiple}.
The value of $\delta$ is left as an input parameter \citep[as in][]{wang2016high}, but any value between
    the localization error $\epsilon_n$ given in the statement of \Cref{thm:wbsrp}
    and the minimal distance $\Delta$ between change points will do.
    In  the proof of we \Cref{thm:wbsrp} we set $\delta \leq 3 \Delta/32 $.

\begin{theorem}[Consistency of $\mathrm{WBSIP}$]\label{thm:wbsrp}
Let Assumptions~\ref{assume:model} and \ref{assume:phase} hold and 
let $\{(\alpha_m,\  \beta_m) \}_{m=1}^M\subset (0, T)^{ M}$ be a
	collection of intervals whose end points are drawn independently and
	uniformly from $\{1,\ldots, T\}$ and such that $\max_{1\le m \le M}
	(\beta_m -\alpha_m)\le C \Delta$  for an absolute
    constant $C>0$.  
Set
\[
    \epsilon_n =C_1 B^4\log(n) \kappa^{-2},
\]
for a $C_1>0$.
Suppose there exist $c_2, c_3 > 0$, sufficiently small,  such that the input
parameters  $\tau$ and $\delta$ satisfy
	\begin{align}\label{eq:tau WBSRP}
		& B^2 \sqrt {\log(n)}< \tau <c_2\kappa \sqrt {\Delta}, \\
		& \epsilon_n  < \delta \le c_3\Delta. \nonumber
	\end{align}    
	Then the collection of the estimated change points
$\mathcal B=\{\hat \eta_k\}_{k=1}^{\widehat K}$	returned by WBSIP with input
parameters of  $(0, n)$, $\{(\alpha_m,\beta_m)\}_{m=1}^M$, $\tau$ and $\delta$ satisfies	
		\begin{align}
			& \mathbb{P}\Bigl\{  \widehat K =K ; \quad \max_{k=1,\ldots,K} |\eta_k-\hat \eta_k| \le  \epsilon_n \Bigr\} \nonumber \\
			\geq & 1-2n^2 M  n^ {-c}  -n^3 9^p  2n^ {-cp}  -\exp\bigl(\log(n/\Delta)- M\Delta^2/(16 n^2)\bigr) \label{eq-thm2-prob}
	 	\end{align}
	for some absolute constants $c > 0$.
\end{theorem}

\begin{remark}[{\bf The relationships among the constants in \Cref{thm:wbsrp}}]
The choice of the constant $C$ is essentially arbitrary but will affect the
choice of the constants $C_1$, $c_2$ and $c_3$, where $c_2$ and $c_3$ in
particular have to be picked small enough. This dependence can be tracked in the
proof but we refrain from giving further details.
\end{remark}

The above theorem implies that the WBSIP algorithm can estimate the
change points perfectly well, with high probability, with a localization rate upper bounded by
\[
    \frac{\epsilon_n}{n} \preceq \frac{B^4}{\kappa^2} \frac{\log (n)}{n}.
\]
The consistency also relies on choosing a large enough number of random intervals $M$.  It follows from \eqref{eq-thm2-prob} that $M \gtrsim n^2\log(n) \Delta^{-2}$ is required.  
The fact that the dimension $p$ does not appear explicitly in the localization
rate is an interesting, if not perhaps surprising, finding. 
Of course, the dimension does affect (negatively) the
performance of the algorithm through \Cref{assume:phase}: keeping $n$ and
$\Delta$ fixed, a larger value of $p$ implies a larger value of
$\frac{B^4}{\kappa^2}$ in order for that assumption to hold. In turn, this leads
to a larger bound in \Cref{thm:wbsrp}. Furthermore, the dimension $p$ appears in
the probability of the event that WBSIP fails to locate all the change points.
We remark that, for the different problem of high-dimensional mean change point
detection,  \cite{wang2016high} also obtained a localization rate independent of
the dimension: see Theorem 3 there. 
In \Cref{section:lb} below we will
prove that \Cref{assume:phase} is in fact essentially necessary for any
algorithm to produce a vanishing localization rate.


\begin{proof}
Since $\epsilon_n$ is the desired order of localization rate, by induction, it suffices to consider any generic $(s, e) \subset (0, T)$ that satisfies 
	\begin{align*}
		\eta_{r-1} \le s\le \eta_r \le \ldots\le \eta_{r+q} \le e \le \eta_{r+q+1}, \quad q\ge -1, \\
		\max \{ \min \{ \eta_r-s ,s-\eta_{r-1}  \}, \min \{ \eta_{r+q+1}-e, e-\eta_{r+q}\} \} \le \epsilon_n,
	\end{align*}
	where $q = -1$ indicates that there is no change point contained in $(s, e)$.

Note that under \Cref{assume:phase}, $\epsilon_n \le \Delta /4$; it, therefore, has to be the case that for any change point $\eta_p\in (0, T)$, either $|\eta_p -s |\le \epsilon_n$ or $|\eta_p-s| \ge \Delta - \epsilon_n \geq 3\Delta /4 $. This means that $ \min\{ |\eta_p-e|, |\eta_p-s| \}\le \epsilon_n$ indicates that $\eta_p$ is a detected change point in the previous induction step, even if $\eta_p\in (s, e)$.  We refer to $\eta_p\in[s,e]$ as an undetected change point if $ \min\{ \eta_p -s, \eta_p-e\} \ge 3\Delta/4 $.

In order to complete the induction step, it suffices to show that WBSIP$((s, e),
\{(\alpha_m,\beta_m)\}_{m=1}^M, \tau, \delta)$ (i) will not detect any new change
point in $(s,e)$ if
all the change points in that interval have been previous detected, and
	(ii) will find a point $b$ in $(s,e)$ (in fact, in $(s + \delta, e - \delta)$) such that $|\eta_p-b|\le \epsilon_n$ if there exists at least one undetected change point in $(s, e)$.

Let 
	\[
	\{ u_m\}_{m=1}^M= PC(\{W_i\}_{i=1}^n, \{(\alpha_m ,\beta_m)\}_{m=1}^M  ).
	\]
	Since the intervals $\{(\alpha_m,\beta_m)\}_{m=1}^M$ are generated
	independently from $\{X_i\}_{i=1}^{n}\cup \{W_i\}_{i=1}^{n}$, the rest
	of the argument is made on the event $\mathcal M$, which is defined in \Cref{event-M} of
	\Cref{section:probability}, and which has no effects on  the distribution of  $\{X_i\}_{i=1}^{n}\cup \{W_i\}_{i=1}^{n}$.
\vskip 3mm

\noindent{\bf Step 1.} 
Let $\lambda_1 = B^2\sqrt{p\log(n)}$.  In this step, we are to show that, on the
event $\mathcal A_1(\{W_i\}_{i=1}^n ,\lambda_1)$ and for some $c'_1>0$, 
	\begin{align} \label{eq:properties of vectors}
		\sup_{1\le m\le M}|u_m^{\top} (\Sigma_{\eta_k} -\Sigma_{\eta_{k-1}}) u_m| \ge c_1'\| \Sigma_{\eta_k} -\Sigma_{\eta_{k-1}}\|_{\op} = c_1'\kappa_k \quad \text{ for every} \quad  k=1,\ldots, K+1
	\end{align}
	On the event $\mathcal M$, for any $\eta_k\in (0, n)$, without loss of
	generality, there exists $\alpha_k\in
	[\eta_k-3\Delta/4,\eta_k-\Delta/2]$ and $\beta_k \in
	[\eta_k+\Delta/2,\eta_k+3\Delta/4]$.  Thus $[\alpha_k, \beta_k]$
	contains only one change point $\eta_k$.  Using \Cref{lemma:lower bound of
	CUMSUM} in \Cref{sec-pre-1} and the inequality $ p\log(n)\le
	\Delta/8$, we have that
		\begin{equation}\label{eq:wbsrp size of population 1}
			\max_{ t = \lceil \alpha_k+\delta  \rceil,\ldots, \lfloor  \beta_k-\delta \rfloor}\| \widetilde \Sigma_t^{\alpha_k,\beta_k}\|_{\op} = \| \widetilde \Sigma^{\alpha_k,\beta_k}_{\eta_k}\|_{\op}  \ge (1/2) \|\Sigma_{\eta_k} -\Sigma_{\eta_{k-1}}\|_{\op} \sqrt{\Delta}.
		\end{equation}
	Let $b_k \in \argmax_{ t = \lceil \alpha_k+\delta  \rceil,\ldots, \lfloor  \beta_k-\delta \rfloor}\| \widetilde S_t^{\alpha_k,\beta_k}\|_{\op},$ where $\widetilde S_t^{s,e}$ denote the covariance CUSUM statistics of $\{ W_i\}_{i=s+1}^e $ at evaluated $t$.  Since $\|\Sigma_{\eta_k} -\Sigma_{\eta_{k-1}}\|_{\op}  = \kappa_k$, by definition,
		\begin{align*}
		 | u_k^{\top} \widetilde \Sigma^{\alpha_k,\beta_k}_{b_k}u_k| & \ge
		 | u_k^{\top} \widetilde S^{\alpha_k,\beta_k}_{b_k}u_k|
		 -\lambda_1\\
		&  = \max_{ t = \lceil\alpha_k+\delta \rceil  ,\ldots, \lfloor  \beta_k-\delta \rfloor }\| \widetilde S_t^{\alpha_k,\beta_k}\|_{\op}-\lambda_1 \\
		& \ge \max_{ t = \lceil\alpha_k+\delta \rceil  ,\ldots, \lfloor
		\beta_k-\delta \rfloor  }\| \widetilde
		\Sigma_t^{\alpha_k,\beta_k}\|_{\op}-2\lambda_1\\
		& \ge (1/2) \|\Sigma_{\eta_k} -\Sigma_{\eta_{k-1}}\|_{\op}
		\sqrt{\Delta} -2\lambda_1\\
		& \ge(1/4)\|\Sigma_{\eta_k} -\Sigma_{\eta_{k-1}}\|_{\op}  \sqrt{\Delta}
		\end{align*}
		where the first and second inequalities hold on the event
		$\mathcal A_1(\{W_i\}_{i=1}^n ,\lambda_1)$, the third inequality
		follows from \eqref{eq:wbsrp size of population 1} and the last
		inequality from \Cref{assume:phase}.
		Next,  observe that 
		\[
		\widetilde \Sigma_t^{\alpha_k, \beta_k} =
			\begin{cases}
				\sqrt {\frac {t-\alpha_k}{(\beta_k-\alpha_k)(\beta_k-t)} }(\beta_k-\eta_k) (\Sigma_{\eta_k} -\Sigma_{\eta_{k-1}}), & t\le \eta_k, \\
				\sqrt {\frac {\beta_k-t}{(\beta_k-\alpha_k)(t-\alpha_k)} }(\eta_k-\alpha_k) (\Sigma_{\eta_k} -\Sigma_{\eta_{k-1}}), &  t\ge\eta_k.
			\end{cases}
		\]
Using the above expression, for $b_k\ge \eta_k$, we have that
	\begin{align*}
		 (1/4) \|\Sigma_{\eta_k} -\Sigma_{\eta_{k-1}}\|_{\op} \sqrt{\Delta}
         & \le \bigl| u_k^{\top} \widetilde \Sigma_{b_k}^{\alpha_k, \beta_k}
         u_k  \bigr|\\
         & = \sqrt {\frac {\beta_k-b_k}{(\beta_k-\alpha_k)(b_k-\alpha_k)} }(\eta_k-\alpha_k)  \bigl | u^{\top}_k (\Sigma_{\eta_k} -\Sigma_{\eta_{k-1}}) u_k \bigl  | \\
		& \le   \sqrt{
		    \frac{(\beta_k-\eta_k)(\eta_k-\alpha_k)}{\beta_k-\alpha_k}}
		    \bigl|u_k^{\top} (\Sigma_{\eta_k} -\Sigma_{\eta_{k-1}}) u_k
		    \bigr|\\
		    & \le \sqrt{2\Delta} \bigl|u_k^{\top} (\Sigma_{\eta_k}
		    -\Sigma_{\eta_{k-1}}) u_k \bigr|.
	\end{align*}
	Therefore \eqref{eq:properties of vectors} holds with $c_1 '= 1/(2\sqrt
	2)$. The case of $b_k < \eta_k $ follows from very similar calculations.
\vskip 3mm

\noindent{\bf Step 2.}
In this step, we will show that WBSIP$((s, e), \{(\alpha_m,\beta_m)\}_{m=1}^M,
\tau,\delta)$ will consistently detect or reject the existence of undetected
change points within $(s, e)$, provided that \eqref{eq:properties of vectors}
holds and on the two events $\mathcal B_1(\{X_i\}_{i=1}^n , \{ u_m\}_{m=1}^M,
\lambda_2)$, where $\lambda_2 = B^2\sqrt{\log(n)}$, and $\mathcal M$, given in
\Cref{eq:B1}
 and \Cref{event-M} in \Cref{section:probability},
respectively.

Let $a_m, b_m$ and  $m^*$ be defined as in WBSIP$((s,e),
\{(\alpha_m,\beta_m)\}_{m=1}^M, \tau,\delta)$.  Denote $Y_i(u_m)= (u_m^{\top}
X_i)^2$ and  $f_i(u_m) = u_m^{\top} \Sigma_i u_m$.  Let  $ \widetilde
Y_{t}^{s,e}(u_m)$ and  $\widetilde f_{t}^{s,e} (u_m)$ be defined as in
\eqref{eq:project sample} and \eqref{eq:project population} of \Cref{sec:properties.cusum} respectively.

Suppose there exists a change point $\eta_p\in (s, e)$ such that $ \min\{ \eta_p
-s, e-\eta_p\} \ge 3\Delta/4 $.  Let $\delta  \le 3\Delta/32$.  Then, on the
event $\mathcal M$, there exists an interval $(\alpha_m,\beta_m)$ selected by
WBSIP such that $\alpha_m \in [\eta_p-3\Delta/4, \eta_p -\Delta/2] $ and  $ \beta_m \in [\eta_p+\Delta/2, \eta_p +3\Delta/4]$.

Then $[s_m', e_m'] = [\alpha_m,\beta_m]\cap [s,e]$, and $[s_m, e_m] = [s_m' +
\delta, e_m' - \delta]$ (see details of the WBSIP procedure in
\Cref{algorithm:WBSRP}). Moreover, we have that $\min \{\eta_p -s_m,e_m-\eta_p\}
\ge (1/2)\Delta$.  Thus, $[s_m,e_m] $ contains at most one change point of the
time series $\{f_i(u_m)\}_{i=1}^n$. A similar calculation as the one shown in
the proof of   \Cref{lemma:lower bound of CUMSUM} gives that 
	\[
	\max_{ \lceil s_m+ \log(n) \rceil \leq t \leq  \lfloor e_m-  \log(n) \rfloor} |\widetilde f _{t}^{s_m,e_m} (u_m)|  \ge (1/8)\sqrt{\Delta} |u_m^{\top} (\Sigma_{\eta_p} -\Sigma_{\eta_{p-1}}) u_m|,
	\]
	where $e_m-s_m\le (3/2)\Delta$ is used in the last inequality.  Therefore
		\begin{align*}
			 a_m &  = \max_{ \lceil s_m+ \log(n) \rceil \leq t \leq 
			 \lfloor e_m- \log(n) \rfloor} | \widetilde
			 Y_{t}^{s_m,e_m} (u_m) |\\
			 & \ge 	\max_{\lceil s_m+ \log(n) \rceil \leq t \leq \lfloor e_m- \log(n) \rfloor}  |  \widetilde f_{t}^{s_m,e_m} (u_m) |  -\lambda _2 \\
			 & \ge   (1/8)\sqrt{\Delta} |u_m^{\top} (\Sigma_{\eta_p}
			 -\Sigma_{\eta_{p-1}}) u_m| -\lambda_2,
		\end{align*}
	where the first inequality holds on the event $\mathcal
	B_1(\{X_i\}_{i=1}^n , \{ u_m\}_{m=1}^M, \lambda_2)$.  Thus for any
	undetected change point $\eta_p$ within $(s, e)$, it holds that
		\begin{align}
			a_{m^*} & = \sup_{1\le m\le  M} a_m \\
			& \ge \sup_{1\le m\le M} (1/8)\sqrt{\Delta} |u_m^{\top}
			(\Sigma_p -\Sigma_{p-1}) u_m| -\lambda_2 \ge (c_1'/8)
			\kappa_p \sqrt \Delta -\lambda_2 \\
			& \ge (c'_1/16)\kappa_p \sqrt \Delta   \label{eq:wbsrp size of population}
		\end{align}
	where the second inequality follows from \eqref{eq:properties of
	vectors}, and the  last inequality from 
		\[
		\lambda_2 = B^2\sqrt{\log(n)} \le (c_1'/16)\kappa \sqrt \Delta,
		\]
		by choosing the constant $C$ in \Cref{assume:phase} to be at
		least $4\sqrt{2}$.

Then, WBSIP$((s, e ), \{(\alpha_m,\beta_m)\}_{m=1}^M, \tau,\delta)$  correctly
accepts the existence of undetected change points on the events \eqref{eq:properties of vectors}, $\mathcal B_1(\{X_i\}_{i=1}^n , \{ u_m\}_{m=1}^M, \lambda_2,\delta)$ and $\mathcal M$.

Suppose there does not exist any undetected change points within $(s, e)$, then for any $(s_m', e_m') = (\alpha_m, \beta_m) \cap (s, e)$, one of the following situations must hold.
	\begin{itemize}
	\item [(a)]	There is no change point within $ (s_m', e_m')$;
	\item [(b)] there exists only one change point $\eta_{r} $ within $(s_m, e_m)$ and $\min \{ \eta_{r}-s_m' , e_m'-\eta_{r}\}\le \epsilon_n $; or
	\item [(c)] there exist two change points $\eta_{r} ,\eta_{r+1}$ within $(s_m, e_m)$ and $\max \{ \eta_{r}-s_m ', e_m'-\eta_{r+1}\}\le \epsilon_n $.
	\end{itemize}

Observe that if (a) holds, then, on the event $\mathcal B_1(\{X_i\}_{i=1}^n , \{
u_m\}_{m=1}^M, \lambda_2)$ given in 
\Cref{eq:B1}
 in \Cref{section:probability}, for $(s_m, e_m) = (s_m'+\delta, e_m'-\delta)$, we have
	\[
	\max_{ \lceil s_m+ \log(n) \rceil \leq t \leq  \lfloor e_m- \log(n) \rfloor}  |  \widetilde Y_{t}^{s_m,e_m} (u_m)|  \le
	\max_{ \lceil s_m+ \log(n)\rceil \leq t \leq \lfloor e_m- \log(n) \rfloor}| \widetilde f_{t}^{s_m,e_m} (u_m) |+\lambda_2 = 0+\lambda_2.
	\]

If (b) or (c) holds, then since $(s_m, e_m) = (s_m' + \delta, e_m' - \delta)$
and $\delta\le \epsilon_n$, it must be the case that $(s_m, e_m)$ does not
contain any change points.  This reduces to case (a). Therefore if \eqref{eq:tau
WBSRP} holds, then WBSIP$((s, e), \{(\alpha_m,\beta_m)\}_{m=1}^M, \tau,\delta)$
will always correctly reject the existence of undetected change points, on the event $\mathcal B_1(\{X_i\}_{i=1}^n , \{ u_m\}_{m=1}^M, \lambda_2)$. 

\vskip 3mm
\noindent {\bf Step 3.} 
Assume that there exists a change point $\eta_p\in (s, e)$ such that $ \min\{ \eta_p -s, \eta_p-e\} \ge 3\Delta/4$.  Let $a_m, b_m$ and $m^*$ be defined as in WBSIP$((s, e), \{(\alpha_m,\beta_m)\}_{m=1}^M, \tau)$. 

To complete the proof it suffices to show that, on the  events $ \mathcal
B_1(\{X_i\}_{i=1}^n , \{ u_m\}_{m=1}^M, \lambda_2) $ and $\mathcal
B_2(\{X_i\}_{i=1}^n , \{ u_m\}_{m=1}^M, \lambda_2)$ given in \Cref{eq:B1} and
\Cref{eq:B2} respectively of \Cref{section:probability}, there exists a change point $\eta_k\in[s_{m*},e_{m*}]$ such that $ \min\{ \eta_k -s, \eta_k-e\} \ge 3\Delta/4 $ and $|b_{m*}-\eta_k|\le \epsilon_n$.

Consider the univariate time series $\{Y_i (u_{m*})\}_{i=1}^n$ and $\{f_i
    (u_{m*})\}_{i=1}^n$ defined in \eqref{eq:project sample} and
    \eqref{eq:project population} of \Cref{sec:properties.cusum}.  Since the
    collection of the change points of the time series $\{f_{i} (u_{m*}
)\}_{i=s_{m*}}^ {e_{m*}}$ is a subset of that of $\{\eta_{k}\}_{k=0}^{K+1}\cap
[s,e]$, we may apply \Cref{coro:wbs 1d} to the time series $\{Y_{i}
(u_{m*})\}_{i=s_{m*}}^{e_{m*}}$ and $\{f_{i} (u_{m*})\}_{i=s_{m*}}^{e_{m*}}$. 
To that end, we will need to ensure that the assumptions of 
 \Cref{coro:wbs 1d}  are verified.
Let $\delta'=\log(n)$ and $\lambda = \lambda_2$.  Observe that \eqref{eq:coro
wbs noise} and \eqref{eq:coro wbs short spacing} are  straightforward
consequences of  \Cref{assume:phase}, \eqref{eq:coro wbs noise 1} and
\eqref{eq:coro wbs noise 2} follow from the definiiton of $ \mathcal B_1(\{X_i\}_{i=1}^n , \{ u_m\}_{m=1}^M, \lambda_2) $ and $\mathcal B_2(\{X_i\}_{i=1}^n , \{ u_m\}_{m=1}^M, \lambda_2)$, and that \eqref{eq:coro wbs size of sample} follows from \eqref{eq:wbsrp size of population}.

Thus, all the conditions in \Cref{coro:wbs 1d} are met, and we therefore
conclude that there exists a change point $\eta_{k}$,  which is also a change
point of $\{f_i(v)\}_{i=s_{m^*}}^{e_{m^*}} $, satisfying
\begin{equation}
\min \{e_{m^*}-\eta_k,\eta_k-s_{m^*}\}    >  \Delta /4 \label{eq:coro wbsrp 1d re1}
\end{equation}
and
\[
		| b_{m*}-\eta_{k}|\le  \max \{ C_3\lambda_2^2 \kappa^{-2}
	    ,\delta' \} \le \epsilon_n,
	    \]
	where the last inequality holds because $\lambda_2^2 \kappa^{-2}  =  B^4
	\log(n) \kappa^{-2}    \ge  \log(n)$, which is a consequence of the
	inequality $B^2 \ge \kappa$.

The proof is complete with  the following two observations: i) The change points
of $\{f_i(u_{m^*})\}_{i=s}^e $ belong to $(s, e)\cap \{ \eta_k\}_{k=1}^K$; and
ii) \Cref{eq:coro wbsrp 1d re1}  and  $(s_{m^*}, e_{m^*}) \subset (s, e)$ imply that
	\[
	\min \{e-\eta_k,\eta_k-s\}  >  \Delta /4 >\epsilon_n.
	\]
	As discussed in the argument before {\bf Step 1}, this implies that
	$\eta_k $ must be an undetected change point of $\{X_i\}_{i=1}^n$ in the covariance structure.
\end{proof}

%% file: LowerBounds.tex

In this section, we provide  lower bounds for the problem of change point
estimation with high dimensional covariance matrices. 

 In \Cref{thm:wbsrp}  we showed that
  if the distribution of $\{X_i\}_{i=1}^n $ satisfies \Cref{assume:model} and
  additionally the condition that $\Delta \ge C B^4\kappa^{-2} p\log(n)$ for
  sufficiently large $C$ as given in \Cref{assume:phase}, then the  WBSIP algorithm
  can, with high probability, detect all the change points with a localization
  rate of the order 
  \begin{equation}\label{eq:loc.rate}
      \frac{B^4}{\kappa^2} \frac{\log (n)}{n}.
  \end{equation}
\Cref{assume:phase} 
might seem a bit arbitrary at first glance.
However, we will show in the next result that if the class of distribution of interest allows 
for a  spacing parameter $\Delta =cB^4\kappa^{-2} p$ for some sufficiently small
constant $c$, then it is \textit{not} possible to estimate the location of the change point at a
vanishing rate.  This result further implies that 
the WBSIP algorithm is optimal in the sense of
requiring the minimax scaling for the problem parameters.

To that effect, we will consider the following class of data generating distribution.   
\begin{assumption}\label{assume:model 2}
	Let $X_1,\ldots, X_n \in \mathbb{R}^p$ be independent Gaussian random vectors such that $X_i \sim N_p ( 0,\Sigma_i)$ 
	with $\| \Sigma_i\|_{\op} \le 2 \sigma $. Let  $\{\eta_k\}_{k=0}^{K+1} \subset \{0, \ldots, n\}$ be a collection of change points, such that  $\eta_0=0 $ and $\eta_{K+1}=n$ and that  
	$$
	\Sigma_{\eta_{k} +1} =\Sigma_{\eta_{k} +2} = \ldots  =\Sigma_{\eta_{k+1}},  \text{ for any } k = 1, \ldots, K+1.
	$$
Assume there exist parameters $ \kappa =\kappa (n)$  and $\Delta=\Delta(n)$ such that
	\begin{align}
		& \inf_{k = 1, \ldots, K+1} \{\eta_k-\eta_{k-1}\}\ge\Delta > 0\nonumber,\\
		& \|\Sigma_{\eta_k} -\Sigma_{\eta_{k-1} } \|_{\mathrm{op}} = \kappa_{k} \geq \kappa > 0, \text{ for any }  k = 1, \ldots, K+1.\nonumber 
	\end{align}
\end{assumption}
We will denote with $P_{\kappa,\Delta,\sigma}^n $ the joint distribution of the
data when the above assumption is in effect.
Notice that $P_{\kappa,\Delta,\sigma}^n$ satisfies \Cref{assume:model} with $B=8
\sigma$. Below we prove that a consistent estimation of the locations of the
change points requires $\frac{\Delta \kappa^2}{p B^4}$ to
diverge, as in
\Cref{assume:phase}. We recall that all the parameters of interest, $\Delta$,
$\kappa$, $\sigma$ and $p$ are allowed to change with $n$. 
The proof is based on a construction
used in
\cite{cai2013optimal} to obtain minimax lower bounds for a class of hypothesis testing
problems involving covariance matrices.

\begin{lemma}\label{lemma:lower bound 1}
Let $\{X_i\}_{i=1}^n$ be a time series satisfying \Cref{assume:model 2} with only one change point and  let $P_{\kappa,\Delta,\sigma}^n $
denote the corresponding joint distribution. Consider the class of distributions 
\[
\mathcal P = \mathcal P(n) =\left\{ P^n_{\kappa,\Delta,\sigma}:   \frac{\sigma^4p}{33\kappa^2} \le  \Delta \le n/3,  \kappa\le \sigma^2/4 \right\},
\]
then
\[
    \inf_{\hat \eta} \sup_{P\in \mathcal P } E_P(|\hat \eta -\eta|) \ge  n/6.
\]
\end{lemma}

\begin{remark}
	
	We remark that while the proof of \Cref{lemma:lower bound 2} is based on Le Cam's lemma,  the phase transition effects concern with the  minimal conditions for {\bf consistent localization} and therefore tells us  a  different story 
		from optimal change point testing. To be more  specific, suppose we know ahead that there can be at most one change point within the internal $[1,n]$. 
		For the problem of  testing, we are asked to distinguish 
		$$
		H_0 : X_i \sim F_0  \text{ for all $i\in [1,n]$} \quad \text{v.s.} \quad 
		H_1: \text{for some $\eta$, } 
		X_i \sim \begin{cases}
		F_0 \quad\text{when} \quad  i\le \eta ;
		\\
		F_1 \quad \text{when} \quad  i\ge \eta+1 ;
		\end{cases}
		$$
		where $F_0 \not = F_1$. On the other hand, consistent localization of change points requires us to not only distinguish $H_0$ and $H_1$, but also  to consistently estimate $\eta$.  
		\\
		In terms of minimax optimality conditions, the one required by  consistent localization (which we summarize as the phase transition effects) should always be  stronger than that by the optimal testing.
\end{remark}
\begin{remark}[{\bf Phase transition for the localization rate in covariance change
    point detection problem}]\label{remark:phase transition}

In light of \Cref{lemma:lower bound 1} and \Cref{thm:wbsrp}, we conclude that,
in the covariance change point detection  problem, the solution undergoes  a
phase transition, which is able to characterize up to logarithmic factor (in
$n$). Specifically,  
\begin{itemize}
    \item if $ \Delta \ge  CB^4p\log(n) /\kappa^2 $ for a sufficiently
    large constant $C > 0$, then it is possible to estimate the locations of the
    change points with a localization rate vanishing in $n$;
\item on the other hand, if  $ \Delta =  cB^4p /\kappa^2 $ for a
    sufficiently small constant $c > 0$, then the localization rate of any
    algorithm remains, at least in
    the worst case, bounded away from $0$.
\end{itemize}
To the best of our knowledge, this phase transition effect is new and unique
in our  settings.
\end{remark}

We conclude this section by showing that the localization rate that we have obtained
for the
WBSIP algorithm, given above in \eqref{eq:loc.rate}, is up to a logarithmic factor, minimax optimal. 
\begin{lemma}\label{lemma:lower bound 2}
	Consider the class of distributions 
	$
	\mathcal Q =\left\{ P^n_{\kappa,\Delta,\sigma}:     \Delta < n/3, \kappa\le \sigma^2/16,  \text{ and } \Delta \kappa^2 \ge  p \log(n)\sigma^4 \right\}.
	$
	Then,
	\[
	\inf_{\hat \eta} \sup_{P\in \mathcal Q } E_P(|\hat \eta -\eta|) \ge  c\sigma^4\kappa^{-2}.
	\]
\end{lemma}
We remark that the class of distribution $\mathcal Q$  in \Cref{lemma:lower bound 2} is a subset of $\mathcal P$ in \Cref{lemma:lower bound 1}. Therefore it is not surprising to see that the lower bound obtained in \Cref{lemma:lower bound 2} is smaller than that in \Cref{lemma:lower bound 1}.

%

%% file: Discussion.tex

In this paper, we tackle the problem of change point detection for a time
series of length $n$ of independent $p$-dimensional random vectors with covariance matrices that are
piecewise constant. We allow the dimension, as well as other
parameters quantifying the difficulty of the problem, to change with $n$.  We
have devised two procedures based on existing algorithms for change point
detection  -- binary segmentation and wild binary
segmentation -- and show that the localization rates they yield are consistent
with those in the univariate time series mean change point detection problems.
In particular we demonstrate the algorithm WBSIP, which applies wild binary
segmentation to carefully chosen univariate projections of the data, produces a
localization rate that is, up to a logarithmic factor,  minimax optimal.

The model setting adopted in the paper allows for the dimension $p$ to grow with
$n$.
However, in order for the localization rates of any procedure to vanish with
$n$ it must be the case that $p$ is of smaller order than $n$. One possible future direction is to consider
different high dimensional settings whereby  $p$ is permitted to grow even
faster than $n$, with additional structural assumptions on the underlying
covariance matrices. For instance, we may model the covariance matrices as spiked matrices with sparse leading eigenvectors.  Another possible modification is to apply the entry-wise maximum norm instead of the operator norm to the covariance CUSUM statistics.  If the changes are still characterized in the operator norms, then this modification requires more careful handling and potentially additional assumptions.

A common, undesirable feature of the WBSIP algorithms is the fact that, for
given interval $(s,e)$, the search of the next change point is limited to
points inside the interval that are $\delta$ away from the endpoints, where
$\delta$ is an input parameter that is larger than the localization error. Such
restriction, which appears also in the algorithm of \cite{wang2016high}
for mean change point localization of high-dimensional time series, is made in
order to prevent the algorithms from returning spurious change points in the
proximity of a true change point. The reason for such phenomenon is subtle, and is ultimately due to the fact that the rate of decay of the expected
value of the covariance CUSUM
statistics around the true change points is in general not uniform, as it
depends on the magnitude of the change. A possible solution to such an issue --
which appears to be unavoidable -- would be to design an adaptive algorithm yielding local rates, one for each
change point. We will pursue this line of research in future work. 

Another undesirable feature of the WBSIP algorithms is the data splitting.  The reason we adopt data splitting here is purely technical that the independence between the projection directions from one half of the data and the second half makes the proof easier.  This might be unnecessary by more involving proofs.

One key assumption used in this paper is the the time series of comprised of
independent observations. This is of course a rather strong condition, which 
might not apply to many real-life problems. In order to handle time dependence,
a natural approach is introduce mixing conditions and/or functional dependency \citep[see, e.g.][]{Wu2005} or assume the observations
come from certain well-defined time series models. 
Further extensions that will be worth pursued include the cases in which the model
is mis-specified or  the observations are contaminated.  
We leave these interesting
extensions to future work but expect that many of the
results derived in this manuscript will provide the theoretical underpinning
for devising and studying more complicated algorithms.  

Another assumption we used throughout the paper is the observations being mean-zero.  One remedy is we could use sample covariance matrix subtracting sample means in the covariance CUSUM statistics, but this requires further assumptions on the variations of the means.  It will be very interesting to develop methods tackling high-dimensional change points in both mean and variance simultaneously, but this is out of the scope of this paper.

%
\section*{Acknowledgments}

We thank Dr Haeran Cho for her constructive comments.

%% file: app.tex
\section{Probabilistic bounds}
\label{section:probability}
\input{appendixa}

\section{Properties of the univariate CUSUM statistics}\label{section:cusum 1d}
\input{appendixb}

\section{Proofs of the Results from \Cref{section:lb}}

\begin{proof}[{\bf Proof of \Cref{lemma:lower bound 1}}]
For any vector $u\in \mathbb R^p$, denote  $\widetilde  \Sigma_{u} = \sigma^2 I_p+\kappa u u^{\top}$. Observe that
if $\kappa \le \sigma^2/4, $ $\| \widetilde \Sigma_{u}\|_{\op} = \sigma^2 +\kappa \le 2 \sigma^2$.

\vskip 3mm

\noindent {\bf Step 1.}  Let $\widetilde  P_{0,u}^n$ denote the joint distribution of independent random
vectors $\{X_i\}_{i=1}^n$ in $\mathbb{R}^p$ such  that 
$$X_1,\ldots,X_{\Delta} \stackrel{i.i.d.}{\sim} N_p(0,\widetilde  \Sigma_u)
\quad \text{and} \quad X_{\Delta+1},\ldots, X_n \stackrel{i.i.d.}{\sim} N_p(0,\sigma^2I).$$ 
Similarly, let 
$\widetilde   P_{1,u}^n$ denote the joint distribution of independent random
vectors  $\{X_i\}_{i=1}^n$ in $\mathbb{R}^p$ with
$$X_1,\ldots,X_{n-\Delta} \stackrel{i.i.d.}{\sim} N_p(0, \sigma^2 I)  \quad
\text{and} \quad X_{n-\Delta+1},\ldots, X_n \stackrel{i.i.d.}{\sim} N_p(0,\widetilde  \Sigma_u).$$ 
Let $ \widetilde  P_i^n=\frac{1}{2^p} \sum_{u\in \{\pm 1\}^{p}/\sqrt{p}} \widetilde  P_{i,u}^n $.
Let $\eta(P_{i,u}^n)$ denote the location of the change point associated to
the distribution $\widetilde  P_{i,u}^n$.
Then 
since $ \eta (\widetilde  P_{0,u}^n)=\Delta$ and 
$ \eta ( \widetilde  P_{1,u}^n)=n-\Delta$ for any $u\in\{\pm 1\}^p/\sqrt{p}$,
$ |\eta (\widetilde  P_{0,u}^n)- \eta (\widetilde  P_{1,u}^n)|\ge n/3 $.
By Le Cam's lemma \citep[see, e.g.][]{yu1997assouad},
\[
\inf_{\hat \eta} \sup_{P\in \mathcal P_{\kappa,\Delta,\sigma}^n } \mathbb{E}_P(|\hat \eta -\eta|) \ge 
(n/3)(1-d_{TV}(\widetilde  P_0^n,\widetilde  P_1^n)),
\]
where $d_{TV}(\widetilde  P_0^n, \widetilde P_1^n) =\frac{1}{2}\| \widetilde  P_0^n- \widetilde  P_1^n\|_1 $.

Let $  \Sigma_{u} =  I_p+ \widetilde \kappa u u^{\top}$, where $\widetilde  \kappa= \kappa / \sigma^2$.  Observe that by assumption 
 $\widetilde \kappa\le 1/ 4.$
Denote with $P_{0,u}^n$ and $ P_{1,u}^n$ the joint distributions of independent
samples $\{X_i\}_{i=1}^n$ in $\mathbb{R}^p$ where
$$X_1,\ldots,X_{\Delta} \stackrel{i.i.d.}{\sim} N_p(0,  \Sigma_u) \quad
\text{and} \quad  X_{\Delta+1},\ldots, X_n \stackrel{i.i.d.}{\sim} N_p(0,I)$$
and
$$X_1,\ldots,X_{n-\Delta} \stackrel{i.i.d.}{\sim} N_p(0,  I)  \quad \text{and}
\quad X_{n-\Delta+1},\ldots, X_n \stackrel{ii.d.}{\sim} N_p(0,  \Sigma_u),$$
respectively.
Since total variation distance is invariant under  rescaling of the covariance,
then $\| \widetilde  P_0^n- \widetilde  P_1^n\|_1=\|   P_0^n-   P_1^n\|_1$. Therefore
\begin{equation}\label{eq:Lecam}
\inf_{\hat \eta} \sup_{P\in \mathcal P_{\widetilde \kappa,\Delta,\sigma}^n } \mathbb{E}_P(|\hat \eta -\eta|) \ge 
(n/3)(1-d_{TV}(  P_0^n, P_1^n))
\end{equation}

\vskip 3mm

\noindent {\bf Step 2.}  Let $x= (x_1,\ldots,x_\Delta)$,  $y=(x_{\Delta+1},\ldots, x_{n-\Delta})$
and $z= (x_{n-\Delta+1},\ldots,  x_n)$.
Let 
\begin{itemize}
\item[$\bullet$]$f_0(x)$ denote the joint distribution of  $X_1,\ldots,X_{\Delta} \sim N_p(0,I)$ and 
$f_u(x)$ denote the joint distribution of  $X_1,\ldots,X_{\Delta}
\stackrel{i.i.d.}{\sim} N_p(0, \Sigma_u)$;
\item[$\bullet$]$g_0(x)$ denote the joint distribution of  $X_{\Delta+1},\ldots,X_{n-\Delta} \stackrel{i.i.d.}{\sim} N_p(0,I)$;
\item[$\bullet$]$h_0(x)$ denote the joint distribution of  $X_{n-\Delta+1},\ldots,X_{n} \stackrel{i.i.d.}{\sim} N_p(0,I)$ and 
$h_u(x)$ denote the joint distribution of $X_{n-\Delta+1},\ldots,X_{n} \stackrel{i.i.d.}{\sim} N_p(0,\Sigma_u)$.
\end{itemize}
Then,
\begin{align*}
&\| P_0^n- P_1^n\|_1
\\
= &\int \int \int \left | \frac{1}{2^p}  \sum_{u \in \{\pm 1\}^p/\sqrt p}  f_u (x) g_0(y)h_0(z)   -\frac{1}{2^p}\sum_{u \in \{\pm 1\}^p/\sqrt p}   f_0 (x) g_0(y)h_u(z)  
\right|\, dxdydz
\\
 =&  \int g_0(y)dy \int \int \left| \frac{1}{2^p}  \sum_{u \in \{\pm 1\}^p/\sqrt p}  f_u (x) h_0(z)   -\frac{1}{2^p} \sum_{u \in \{\pm 1\}^p/\sqrt p}  f_0 (x)h_u(z)  
\right|\, dxdz
\\
= &\int \int \left| \frac{1}{2^p}  \sum_{u \in \{\pm 1\}^p/\sqrt p}  f_u (x) h_0(z)   -\frac{1}{2^p} \sum_{u \in \{\pm 1\}^p/\sqrt p}  f_0 (x)h_u(z)  
\right|\, dxdz 
\\
\le&
\int \int \left|\frac{1}{2^p}  \sum_{u \in \{\pm 1\}^p/\sqrt p}  f_u (x) h_0(z)   -f_0(x)h_0(z)\right|  
+ \left |f_0(x) g_0(z)-\frac{1}{2^p} \sum_{u \in \{\pm 1\}^p/\sqrt p}  f_0 (x)h_u(z)  \right |\,dxdz\\
=&2 \| P^\Delta_0 - P^\Delta_1\|_1,
\end{align*}
where $P^\Delta_0$ is the joint distribution of $X_1,\ldots,X_{\Delta} \stackrel{i.i.d.}{\sim} N_p(0,I)$
and $P^\Delta_1 = \frac{1}{2^p} \sum_{u \in \{\pm 1\}^p/\sqrt p}  P^\Delta_{1,u}$, where
$P^\Delta_{1,u}$ is the joint distribution of $X_1,\ldots,X_{\Delta} \stackrel{i.i.d.}{\sim} N_p(0,\Sigma_u)$.
Thus \eqref{eq:Lecam} becomes
\begin{equation}\label{eq:Lecam 2} \inf_{\hat \eta}\sup_{P\in \mathcal P_{\widetilde \kappa,\Delta,\sigma}^n } \mathbb{E}_P(|\hat \eta -\eta|)  \ge  (n/3)(1-\| P^\Delta_0 - P^\Delta_1\|_1).\end{equation}

\vskip 3mm

\noindent {\bf Step 3.}  
To bound $2\|  P_0^\Delta-  P_1^\Delta\|_1 $, let 
$P_0=N_p(0, I_p)$ and $P_u=  N_p(0, \Sigma_{u}) $.
It is easy to see that 
\begin{align*}\chi^2 (P_1^\Delta, P_0^\Delta) 
= \mathbb{E}_{P_0^\Delta}\left( \frac{dP_1^\Delta}{dP_0^\Delta} -1\right)^2
=\frac{1}{4^p} \sum_{u,v\in \{\pm 1\}^p} \mathbb{E}_{P_0^\Delta}\left( \frac{dP_u^\Delta}{dP_0^\Delta} \frac{dP_v^\Delta}{dP_0^\Delta}\right)-1.
\end{align*}
For any $u,v \in \{\pm 1\}^p$,
\begin{align*}
\mathbb{E}_{P_0^\Delta}\left( \frac{dP_u^\Delta}{dP_0^\Delta} \frac{dP_v^\Delta}{dP_0^\Delta}\right)
=\left(\mathbb{E}_{P_0}\left( \frac{dP_u}{dP_0} \frac{dP_v}{dP_0}\right)\right)^\Delta 
=(1-(\widetilde  \kappa u^{\top}v)^2 )^{-\Delta/2},
\end{align*}
where the last equality follows from \Cref{lemma:gaussian property}. 
Denote
 $ U$ and $V$ to be the $p$ dimensional Rademacher variables  with $U$ being independent of $V$ and 
 $\varepsilon_p = \left( 1^{\top}V/p\right)^2$. Thus
\begin{align*}
\chi^2 (P_1^\Delta, P_0^\Delta) &=
\frac{1}{4^p} \sum_{u,v\in \{\pm 1\}^p}(1-(\widetilde \kappa u^{\top}v)^2 )^{-\Delta/2} -1
\\
&= \mathbb{E}_{U,V} \left( \left( 1-(\widetilde  \kappa U^{\top}V/p)^2 \right)^{\Delta/2}\right) -1\\
&= \mathbb{E}_{V} \left( \left( 1-(\widetilde  \kappa 1^{\top}V/p)^2 \right)^{\Delta/2}\right) -1\\
&\le \mathbb{E} \left( \exp( \varepsilon_p \widetilde \kappa^2\Delta) \right)-1,
\end{align*}
where
 the  inequality follows from inequality 
 $(1-t)^{-\Delta/2}\le \exp(\Delta t) $ for any $t\le 1/2 $ and that $\widetilde  \kappa\le 1/\sqrt{2}$.
The Hoeffding's inequality, applied to Rademacher variables, gives
\begin{equation}\label{eq:Hoeffding} P(\varepsilon_p \ge \lambda)\le 2 e^{-2p\lambda}.
\end{equation}
Thus 
\begin{align*}
\mathbb{E} \left( \exp( \varepsilon_p\widetilde  \kappa^2\Delta) \right)
&= \int_0^\infty P\left(\exp \left(\varepsilon_p \widetilde  \kappa^2\Delta \ge u\right)\right) \, du\\
&\le 1 +\int_1^\infty P\left( \varepsilon \ge \log(u)/(\widetilde  \kappa^2\Delta) \right)\, du\\
& \le 1+  \int_1^\infty 2\exp\left(- \frac{\log(u)2p}{\widetilde \kappa^2\Delta} \right)\, du \\
&= 1-  \frac{2}{1-\frac{2p}{\widetilde \kappa^2\Delta}},
\end{align*}
where the second inequality follows from \eqref{eq:Hoeffding}, and the last equality holds if 
$\frac{2p}{\widetilde  \kappa^2\Delta}>1. $ Thus if 
$\Delta = \frac{2p}{33\widetilde  \kappa^2} = \frac{2p\sigma^4}{33\kappa^2} ,$
 then, using the well-known fact that 
$\|  P_0^\Delta-  P_1^\Delta\|_1 \le 2\sqrt{ \chi^2 (P_1^\Delta, P_0^\Delta) }
$, we obtain the bounds
$$\|  P_0^\Delta-  P_1^\Delta\|_1 \le 2\sqrt{ \chi^2 (P_1^\Delta, P_0^\Delta) } \le  2 \sqrt{\frac{2}{\frac{2p}{ \widetilde \kappa^2\Delta}-1}}= 1/2  .$$
This and \eqref{eq:Lecam 2} give
 \begin{equation*}
\inf_{\hat \eta} \sup_{P\in \mathcal P } \mathbb{E}_P(|\hat \eta -\eta|) \ge n/6.
\end{equation*}
\end{proof}

\begin{proof}[Proof of \Cref{lemma:lower bound 2}]
Let 
$$\mathcal Q = \left\{ P^n_{\kappa,\Delta}:     \Delta < n/2 \right\}.$$
Using the same argument as in Step 1 of the proof of  \Cref{lemma:lower bound 1}, it suffices to 
take 
$$\widetilde \kappa= \kappa / \sigma^2 \le 1/4,$$ and consider 
$\Sigma_{u} = I_p+\widetilde \kappa u u^{\top}$.

Let $u$ to be any unit vector in $\mathbb R^p$ and $\delta$ a positive number.
Let $P_{0}^n$ denote the joint distribution of  independent samples $\cX$ in
$\mathbb{R}^p$ where
$$X_1,\ldots,X_{\Delta-1} \stackrel{iid.}{\sim} N_p(0,I) , \quad X_{\Delta},\ldots, X_n  \stackrel{iid.}{\sim} N_p(0,\Sigma_u)$$ 
and 
$ P_{1}^n$ the joint distribution of  independent samples $\cX$ in
$\mathbb{R}^p$ where
$$X_1,\ldots,X_{\Delta+\delta} \stackrel{iid.}{\sim} N_p(0,I) , \quad X_{\Delta+\delta+1},\ldots, X_n \stackrel{iid.}{\sim} N_p(0,\Sigma_u).$$ 
Let $\eta(P_{i}^n)$ denote the location of the change point of distribution $P_{i,u}^n$.
Then 
since $ \eta (P_{0}^n)=\Delta$ and 
$ \eta (P_{1}^n)=\Delta+\delta$, thus 
$ |\eta (P_{0}^n)- \eta (P_{1}^n)|\ge \delta $.
By Le Cam's lemma \citep{yu1997assouad},
\[
\inf_{\hat \eta} \sup_{P\in \mathcal Q } \mathbb{E}_P(|\hat \eta -\eta|) \ge 
\delta(1-d_{TV}(P_0^n,P_1^n)) = \delta(1 -\frac{1}{2} \| P_0^n-P_1^n\|_1) 
\]
Since $X_1,\ldots X_\Delta-1$ and $X_{\Delta+\delta},\ldots, X_{n}$ are identically distributed as $P_0^n$ and $P_1^n$ respectively,
$\| P_0^n-P_1^n\|_1 = \|P_0^\delta- P_1^\delta\|_1$,
where $P_0^\delta$ is the joint distribution of $X_1,X_\delta \sim N_p(0,I)$
and $P_1^\delta$ is the joint distribution of $X_1,X_\delta \sim N_p(0,\Sigma_u)$. 
By \Cref{lemma:gaussian property},
\begin{align*}
\chi^2  (P_1^\delta, P_0^\delta) =  (1 - \widetilde  \kappa^2 ) ^{-\delta/2} -1 \le 4 \widetilde  \kappa^2 \delta.
\end{align*}
if $\widetilde\kappa^2 \delta/2\le 1/4 $ and $\delta/2\ge 2$.
Thus by taking 
$\delta =\widetilde  \kappa^{-2}/4  $ , we have
$$4 \le \delta = \kappa^{-2}\sigma^{4}/4  \le\Delta/(4p \log (n)) \le \Delta/4$$
and that
$$ \inf_{\hat \eta} \sup_{P\in \mathcal Q } \mathbb{E}_P(|\hat \eta -\eta|) \ge  \delta\left(1 - \frac{1}{2} \sqrt{  \chi^2 (P_1^\delta, P_0^\delta) } \right) \ge \widetilde  \kappa^{-2} /32 = \sigma^4\kappa^{-2} /32 .$$

\end{proof}

\begin{lemma}\label{lemma:gaussian property}
Let $P_0 =N_p(0,I_p)$ and $P_u = N_p(0, I_p+\kappa u u')$. Then 
$$ \mathbb{E}_{P_0} \left(\frac{dP_u}{dP_0} \frac{dP_v}{dP_0}\right)  = \left(1- (\kappa u^{\top}v)^2 \right)^{-1/2}.$$
\end{lemma}
\begin{proof}
See, e.g., lemma 5.1 in \cite{berthet2013optimal}.
\end{proof}
\begin{lemma} For $t\ge 2,x\ge 0$, if  $tx \le 1/4 $,
$ (1-x)^{-t}  -1\le 4tx $
\end{lemma}
\begin{proof}
There exists $s\in [0,x]$ such that 
$$ (1-x)^{-t} -1=  tx + t(t+1) x^2(1-s)^{-t-2} \le tx +4 t^2x^2  (1-x)^{-t}  \le 2tx +4 t^2x^2  ((1-x)^{-t}-1) .$$
Thus 
$$ (1-x)^{-t} -1\le \frac{2tx}{1-4t^2x^2} \le 4tx$$
\end{proof}

\section{Properties of the covariance CUSUM statistic}\label{sec-pre-1}
\subsection{Properties of 1d CUSUM statistics}

\begin{lemma}\label{lemma:number of changes}
Suppose $[s, e]\subset [1,T] $ such that $e-s\le C_R\Delta$,
and that 
$$
\eta_{r-1} \le s\le \eta_r \le \ldots\le \eta_{r+q} \le e \le \eta_{r+q+1}, \quad q\ge 0.
$$
Denote
$$\kappa_{\max}^{s,e} =\max \{ \eta_{p} -\eta_{p-1} : r\le p \le r+q\}.$$
Then for any $r-1 \le p \le r+q $,
$$ \left|\frac{1}{e-s}\sum_{i=s}^e f_i - f_{\eta_p} \right| \le C_R\kse.$$
\end{lemma}
\begin{proof}
Since $e-s\le C_R\Delta$, the interval $[s,e]$ contains at most $C_R+1$ change points.
Observe that 
\begin{align*}
 &\left|\frac{1}{e-s}\sum_{i=s}^e f_i - f_{\eta_p} \right|
\\
=
 & \frac{1}{e-s} \left|  \sum_{i=s}^{\eta_r} (f_{\eta_{r-1}} - f_{\eta_p}) +  \sum_{i={\eta_r +1}}^{\eta_{r+1}}(f_{\eta_{r }} - f_{\eta_p})
 + \ldots + \sum_{i={\eta_{r +q}+1}}^{e} (f_{\eta_{r +q}} - f_{\eta_p})
 \right|
 \\
 \le 
 &
  \frac{1}{e-s}  \sum_{i=s}^{\eta_r} |p-r|\kse +  \sum_{i={\eta_r +1}}^{\eta_{r+1}} |p-r-1|\kse
 + \ldots + \sum_{i={\eta_{r +q}+1}}^{e} |p-r-q-1|\kse
\\
\le &
 \frac{1}{e-s} \sum_{i=s}^{e}  (C_R+1)\kse,
 \end{align*}
where $|p_1 -p_2| \le C_R +1 $ for any $\eta_{p_1}, \eta_{p_2} \in [s,e]$ is used in the last inequality.
\end{proof}

\begin{lemma}\label{lemma:one change point basics}
If $\eta_p$ is the only change point in $[s,e]$, then
$$ |\widetilde f^{s,e}_{\eta_p} | = \sqrt { \frac{(\eta_p-s)(e-\eta_p)}{e-s}  } \kappa_p  \le \sqrt { \min \{\eta_p -s , e -\eta_p\} } \kappa_p$$ 
\end{lemma}

\begin{lemma}\label{lemma:cusum boundary bound}
Let $[s,e]$ contains two or more change points such that 
\[
\eta_{r-1} \le s\le \eta_r \le \ldots\le \eta_{r+q} \le e \le \eta_{r+q+1}, \quad q\ge 1.
\]
If 
$$\eta_{r}-s \le  c_1^2\Delta $$
 then
$$|\widetilde f^{s,e}_{\eta_r}| \le c_1  |  \widetilde f^{s,e}_{\eta_{r+1}}| +2\kappa_r  \sqrt {\eta_r -s}. $$
\end{lemma}
This can be useful in testing when there are exactly two change points with 
$$ \eta_r-s \le \lambda^2 \kappa_r^{-2}, \quad e-\eta_{r+1} \le \lambda^2 \kappa_{r+1}^{-2} . $$ 
It is also useful to show $\eta_r -s \ge \Delta/4 $ for some absolute constant $c$ when
$$|\widetilde f^{s,e}_{\eta_r}| \ge \max_{s\le t \le e} |\widetilde f^{s,e}_{t}| -2 \lambda.$$

\begin{proof}

Consider the sequence $\{g_t\}_{t=s+1}^e $ be such that 
$$ 
g_t=
\begin{cases}
f_{\eta_{r+1}}  \quad \text{if} \quad s +1\le  t\le \eta_{r},
\\
f_t \quad \text{if} \quad \eta_{r}+1 \le t \le e. 
\end{cases}
$$
For any  $t\ge \eta_r$,
\begin{align*}
\widetilde f^{s,e}_{\eta_r} - \widetilde g^{s,e}_{\eta_r} 
&=\sqrt { \frac{(e-s)-t}{(e-s)(t-s)} } (\eta_r-s) (f_{\eta_{r+1}} -f_{\eta_{r}})  \le \sqrt{\eta_r-s} \kappa_r
\end{align*}
Thus 
\begin{align*}
|\widetilde f^{s,e}_{\eta_r} | &
\le  |\widetilde g^{s,e}_{\eta_r}  |+  \sqrt{\eta_r-s} \kappa_r
\\
&\le \sqrt { \frac{(\eta_r-s)  (e-\eta_{r+1})  }{   ( \eta_{r+1}-s)  (e-\eta_r)    } }|\widetilde g^{s,e}_{\eta_{r+1}} |+  \sqrt{\eta_r-s} \kappa_r
\\
&\le \sqrt { \frac{c_1^2\Delta }{ \Delta}}|\widetilde g^{s,e}_{\eta_{r+1}} |  +  \sqrt{\eta_r-s} \kappa_r
\\
&\le c_1 |\widetilde f^{s,e}_{\eta_{r+1}} |  + 2\sqrt{\eta_r-s} \kappa_r.
\end{align*}
where the first inequality follows from the observation that the first change point of $g_t$ in $[s,e]$ is at $\eta_{r+1}$.
\end{proof}

\subsection{Properties of the covariance CUSUM statistics}
\label{sec:properties.cusum}
All of our consistency results  heavily rely on the properties of  population
quantity of the CUSUM statistic.  In the covariance change point detection
problem, however, it is not trivial to analyze the properties of the function
$t \mapsto \bigl\|\tSigma_t\bigr\|_{\mathrm{op}}$ in the multiple change point
case.  For example, it is difficult to determine the regions of monotonicity of
$\bigl\|\tSigma_t\bigr\|_{\mathrm{op}}$ as a function of $t$ as is done in
\citet[][Lemma~2.2]{venkatraman1992consistency}.  As a remedy, we introduce the
concept of \emph{shadow vector}, which is defined as a maximizer of the operator
norm of the CUSUM statistics in all the following results.  In this way, we turn the covariance change point detection problem into a mean change point detection problem.   

For any $v\in\mathbb R^p$ with $\|v\|=1$, let $Y_i (v)  = (v^{\top}X_i)^2$ and $f_i (v) = v^{\top}\Sigma_iv$, $i = 1, \ldots, n$.  Note that both $\{Y_i(v)\}$ and $\{f_i(v)\}$ are univariate sequences, we hence have the corresponding CUSUM statistics defined below
	\begin{align}
	\widetilde Y_{t}^{s,e}(v) &=\sqrt{\frac{e-t}{(e-s) (t-s)}}\sum_{i=s+1}^{t}Y_i(v)- \sqrt{\frac{t-s}{(e-s) (e-t)}} \sum_{i=t+1}^{e} Y_i(v) 
	,\label{eq:project sample}\\
	\widetilde f_{t}^{s,e} (v)&=\sqrt{\frac{e-t}{(e-s) (t-s)}}\sum_{i=s+1}^{t}f_i(v)- \sqrt{\frac{t-s}{(e-s) (e-t)}} \sum_{i=t+1}^{e} f_i(v) \label{eq:project population}.
	\end{align}

The key rationale of the CUSUM based BS algorithm or any variants thereof being a powerful tool selecting the change points is that the population version of the CUSUM statistic achieving its maxima at the true change points.  In Lemma~\ref{lemma:maximized at change points}, we show the same holds for the covariance CUSUM statistic.

\begin{lemma}\label{lemma:maximized at change points}
Assume $(s,e)\cap \{\eta_k\}_{k=1}^K \not = \emptyset$ and \Cref{assume:model}.  The quantity $\bigl\|\widetilde \Sigma_t ^{s,e} \bigr\|_{\mathrm{op}}$ as a function of $t$ achieves its maxima at the true change points, i.e.
	$$
	\argmax_{t = s+1, \ldots, e-1} \bigl\|\tSigma_t\bigr\|_{\mathrm{op}} \cap \{\eta_k\}_{k=1}^K \neq \emptyset.
        $$
\end{lemma}
\begin{proof}
For the sake of contradiction suppose that there exists $t^* \in (s,e) \setminus \{\eta_k\}_{k=1}^K$  such that 
	\[
	t^* \in \argmax_{t = s+1, \ldots, e-1} \bigl\| \widetilde \Sigma^{s,e}_{t}\bigr\|_{\mathrm{op}},
	\]
	and
	\[
	\bigl \|\widetilde\Sigma_{t^*}^{s,e} \bigr\|_{\mathrm{op}} > \max_{k:\,  \eta_k\in (s, e)}\bigl\| \widetilde \Sigma_{\eta_k}^{s,e}\bigr\|_{\mathrm{op}}.
	\]
	Let $v\in \argmax_{\|u\|=1} |u^{\top} \widetilde \Sigma_{t^*}^{s,e} u|$, and consider the sequence $\{f_i (v)\}_{i=1}^n = \{ v^{\top}\Sigma_i v\}_{i=1}^n$. By the above display, we have
	\begin{equation}\label{eq:t star} 
	\bigl|\tf_{t^*} (v)\bigr|  =  \bigl\|\tSigma_{t^*} \bigr\|_{\mathrm{op}}  > \max_{k:\, \eta_k \in (s, e)}  \bigl\|\tSigma_{\eta_k} \bigr\|_{\mathrm{op}} \ge \max_{k: \, \eta_k \in (s, e)}  \bigl|\widetilde f_{\eta_k} (v)\bigr|, 
	\end{equation}
	where $\widetilde f_t^{s, e}(v)$ is defined in \eqref{eq:project population}.  It follows from Lemma 2.2 of \cite {venkatraman1992consistency}, the quantities  $\bigl| \widetilde f_t^{s, e}(v) \bigr|$ are maximized at the change points of the time series $\{ f_t(v)\}_{t=s+1}^e$.  Note that the change points of the sequence $\{ f_t(v)\}_{t=s+1}^e$ are a subset of $\{\eta_{k}\}_{k=1}^K$.  This contradicts \eqref{eq:t star}.
\end{proof}

Lemma~\ref{lemma:maximized at change points} shows that the population version of the covariance CUSUM statistic is maximized at the true change points in terms of the operator norm.  In Lemma~\ref{lemma:lower bound of CUMSUM} below, we give the lower bound of the maxima thereof.  One can interpret it as the signal strength.

\begin{lemma}\label{lemma:lower bound of CUMSUM}
Under Assumption \ref{assume:model}, let $0\le s < \eta_k < e \le n $ be any interval satisfying 
$$ \min \{\eta_k-s, e-\eta_k \}\ge c_1\Delta.$$ Then for any $0 <\delta < (c_1/8)\Delta$,
	\[
	\max_{t = \lceil s+\delta \rceil, \ldots, \lfloor e-\delta \rfloor} \| \widetilde \Sigma_{t}^{s,e}\|_{\mathrm{op}}  \ge (c_1/2)\kappa\Delta (e-s)^{-1/2} .
	\]
\end{lemma}
\begin{proof}  Let
	\[
	v \in \argmax_{\|u\|=1} \bigl| u^{\top}\bigl( \Sigma_{\eta_k} - \Sigma_{\eta_{k+1}} \bigr)u\bigr|.
	\]
Therefore
	$
	\bigl\| \widetilde \Sigma^{s,e}_{t}\bigr\|_{\mathrm{op}} \ge \bigl|\widetilde f^{s,e}_{t}  (v)\bigr|.
	$
	Since
	\[
	\bigl|f_{\eta_k} (v) - f_{\eta_{k +1}} (v)\bigr|  = \bigl|v^{\top}  \bigl( \Sigma_{\eta_k}-\Sigma_{\eta_{k+1}} \bigr)v\bigr| = \bigl\|\Sigma_{\eta_k}-\Sigma_{\eta_{k+1}} \bigr\|_{\mathrm{op}}\ge \kappa,
	\]
we have
	\[
	\max_{t = \lceil s+\delta      \rceil, \ldots, \lfloor e-\delta   \rfloor} \| \widetilde \Sigma_{t}^{s,e}\|_{\mathrm{op}} \geq \max_{t = \lceil s+\delta \rceil, \ldots, \lfloor e-\delta \rfloor} |\widetilde f^{s,e}_{t}  (u)| \ge (c_1/2)\kappa\Delta (e-s)^{-1/2},
	\]
	where the second inequality follows from the same arguments in  \citet[][Lemma~2.4]{venkatraman1992consistency} by regarding $\widetilde f^{s,e}_{t}(u)$ as the CUSUM statistic for a univariate time series.
\end{proof}

In Lemma~\ref{lemma:reduction to 1d}, we show that we can conform the problem of change point detection on covariance to the one on mean via shadow vectors.  This translation allows us to convert a high-dimensional covariance problem to a univariate mean problem.

\begin{lemma}\label{lemma:reduction to 1d}
Under the same assumptions as \Cref{lemma:lower bound of CUMSUM}, let
	$$
	b\in \argmax_{t = \lceil s+  \delta         \rceil, \ldots, \lfloor e- \delta     \rfloor} \bigl\| \tS_{t} \bigr \|_{\mathrm{op}},
	$$ and denote a shadow vector by
	\begin{equation}\label{eq:shadow vector}
	v\in \argmax_{\|u\|=1} \bigl|u^{\top} \tS_b u\bigr|.
	\end{equation}
       Then
	$$
	b \in\argmax_{t = \lceil s+  \delta      \rceil, \ldots, \lfloor e- \delta     \rfloor} \bigl| \widetilde Y^{s,e}_{t} (v) \bigr|.
	$$
\end{lemma}
\begin{proof}
It suffices to show that 
	\[
	\bigl|\widetilde Y^{s,e}_{b}(v)\bigr|  =  \bigl\|\tS_b \bigr\|_{\mathrm{op}}  \ge \max_{t = \lceil s+ \delta    \rceil, \ldots, \lfloor e- \delta   \rfloor} \bigl \|\tS_t \bigr\|_{\mathrm{op}} \ge \max_{t = \lceil s+ \delta       \rceil, \ldots, \lfloor e- \delta    \rfloor} \bigl|\widetilde Y^{s,e}_{t}(v)\bigr|. 
	\]
 \end{proof}
 
For the shadow vector $v$ defined in \eqref{eq:shadow vector}, it is tempting to argue that we can estimate $v$ and then use $\{Y_i(v)\}_{i=1}^n$ as the new sequence to derive an upper bound for the rate of localization.  This ideal approach, however, suffers from two non-trivial obstacles.
	\begin{itemize}
	\item We don't have any guarantee that the estimation of $v$ is consistent.  This is because estimating the first principle direction of any sample matrix in general requires a non-vanishing gap between the first and the second eigenvalues of the corresponding  population matrix.  Since $b$ depends on the data, without more involved additional assumption, it is difficult to show $\tS_b$ converges to a population quantity as $n\to \infty$.

	\item Suppose $v$ given in \eqref{eq:shadow vector} has a well defined population quantity, say $\mathbb{E}(v)$.  The estimation of $v$ depends on $\cX$ and therefore $\mathbb{E}(Y_i(v)) \neq  f_i(\mathbb{E}(v))$. 
	\end{itemize}
	On one hand our knowledge of the population version of the shadow vector $v$ is very limited; on the other hand, in Lemma~\ref{coro:shadow properties} below we show that without estimating the population shadow vector, the CUSUM statistics and its corresponding population quantity are close enough  and that the maximum of the CUSUM statistic is detectable.

\begin{lemma}\label{coro:shadow properties}
Under the same assumptions as  \Cref{lemma:lower bound of CUMSUM}, assume
	\[
	\max_{t = \lceil s+\delta \rceil, \ldots, \lfloor e-\delta  \rfloor}
	\bigl\| \tS_t  -\tSigma_t \bigr\|_{\mathrm{op}} \le \lambda,
		\]
	where $\lambda_1 > 0$.  For the shadow vector $v$ defined in \eqref{eq:shadow vector}, then
		\[
	\max_{t = \lceil s+\delta    \rceil, \ldots, \lfloor e -\delta  \rfloor}\bigl|\widetilde Y_t^{s,e} (v)  \bigr| \ge (c_1/2)\kappa\Delta (e-s)^{-1/2}-\lambda. 
	\]
\end{lemma}

\begin{proof}
Observe that for all compatible $t$ and $v$,
	\[
	\bigl|\widetilde Y_t^{s,e} (v)  -\widetilde f_t ^{s,e} (v) \bigr|  = \bigl|v^{\top} \bigl(  \widetilde S_t^{s,e}  - \widetilde \Sigma_t^{s,e}  \bigr) v \bigr| \leq \bigl\| \tS_t  -\tSigma_t \bigr\|_{\mathrm{op}} \leq \lambda.
	\]

In addition, for the shadow vector $v$, we have
	\begin{align*}
	& \max_{t = \lceil s+\delta  \rceil, \ldots, \lfloor e-\delta \rfloor} \bigl|\widetilde Y_t^{s,e} (v) \bigr |  =  \max_{t = \lceil s+\delta  \rceil, \ldots, \lfloor e-\delta \rfloor} \bigl\|\widetilde S_t^{s,e} \bigr\|_{\mathrm{op}} \\
	\geq & \max_{t = \lceil s+\delta \rceil, \ldots, \lfloor e-\delta\rfloor} \bigl\|\tSigma_t \bigr\|_{\mathrm{op}} -\lambda_1 \ge (c_1/2)\kappa\Delta (e-s)^{-1/2}-\lambda,
	\end{align*}
	where the last inequality follows from Lemma~\ref{lemma:lower bound of CUMSUM}.	
\end{proof}

We finally prove a simple property  of covariance CUSUM statistics.
\begin{lemma}\label{lemma:size of boundary}
$ \|\widetilde \Sigma_t^{s,e} \|_{\op} \le 2 \sqrt {\min\{e-t,t-s \}} B^2 $ for any $t\in (s, e)$.

\end{lemma}
\begin{proof}
Observe that 
\begin{align*}
 \|\widetilde \Sigma_t^{s,e} \|_{\op} \le 2\sqrt {\frac{ (e-t)(s-t)}{e-s}} B^2 \le  2 \sqrt {\min\{e-t,t-s \}} B^2 .
\end{align*}

\end{proof}

\end{document}

%% file: appendixa.tex

In this section, we give basic high-probability concentration bounds on the fluctuations of the covariance CUSUM statistics using the notions of sub-Gaussian and
sub-Exponential random vectors.  We also state some properties of the randomly
selected intervals $\{s_m,e_m\}$ in the WBS algorithm, which hold with high
probability.
\vskip 3mm

We start by introducing the definitions of sub-Gaussian and sub-Exponential
random variables through Orlicz norms. See, e.g., \cite{vershynin2010introduction} for more details. 
\begin{definition} \label{def-orlicz}
\begin{itemize}
\item [(i)] A random variable $X\in \mathbb R$ is  sub-Gaussian  if 
$$ \| X\|_{\psi_2} := \sup_{k\ge 1} k^{-1/2}  \{\mathbb E(| X|^k)\}^{1/k} <\infty.  $$
A random vector $X\in \mathbb R^p$ is  sub-Gaussian  if 
$$ \| X\|_{\psi_2} := \sup_{v\in \mathcal S^{p-1} }\| v^{\top}X \|_{\psi_2} <\infty,$$
where $\mathcal S^{p-1}$ denote unit sphere in Eucliden norm in $\mathbb R^p$.
\\
\item [(ii)] A random variable $Y\in \mathbb R$ is  sub-Exponential  if 
$$ \| Y\|_{\psi_1} = \sup_{k\ge 1} k^{-1}  \mathbb E(| Y|^k)^{1/k} <\infty.  $$
A random vector $Y\in \mathbb R^p$ is  sub-Exponential if 
$$ \| Y\|_{\psi_1} = \sup_{v\in \mathcal S^{p-1} }\| v^{\top}Y \|_{\psi_1} <\infty.$$
\end{itemize}
\end{definition}

We note that if $X\in\mathbb R$ is sub-Gaussian, then $X^2$ is sub-Exponential,
due to the easily verifiable fact that 
\begin{equation}\label{eq:bound.orlictz}
    \| X\|^2_{\psi_2} \le \| X^2 \|_{\psi_1}\le 2 \| X\|_{\psi_2}^2.
\end{equation}

Recall the sample and the population versions of the covariance CUSUM statistics
given in \Cref{def-1}.  For $\lambda > 0$, we define the events, which
depend on $\{X_i\}_{i=1}^n$.
	\begin{align}
		& \mathcal A_1 (\{X_i\}_{i=1}^n, \lambda) = \left\{ \sup_{0\le s < t < e \le n  }\left\| \tS_t  -\tSigma_t \right\|_{\mathrm{op}} \le \lambda,\quad  \min\{t-s, e-t\} \ge p\log(n)  \right\} \label{eq-event-A1}\\
		\text{and} &  \nonumber \\
		& \mathcal{A}_2(\{X_i\}_{i=1}^n, \lambda) =   \left \{
		    \sup_{0\le s < e \le n} \frac{\left \| \sum_{i=s+1}^e  (X_i
		    X_i^{\top} - \Sigma_i  ) \right\|_{\op} }{\sqrt{e-s}} \le
		\lambda,  \quad   e-s \ge p\log(n) \right\}. \nonumber
	\end{align}
	Next, for an arbitary collection $\{v_m\}_{m=1}^M$ of deterministic
	unit vectors in $\mathbb{R}^p$ we define the events
	\begin{align}
		& \mathcal B_1 (\{X_i\}_{i=1}^n, \{ v_m\}_{m=1}^M, \lambda)  =
		\left\{ \sup_{1\le m\le M}  \sup_{0\le s < t < e \le n }\Bigl |
		v_m^{\top}( \tS_t  -\tSigma_t) v_m \Bigr| \le \lambda, \quad
		\min\{t-s, e-t\} \ge \log(n)  \right\} \label{eq:B1}\\
		\text{and} & \nonumber \\
		& \mathcal B_2 (\{X_i\}_{i=1}^n, \{ v_m\}_{m=1}^M, \lambda)  =
		\left\{ \sup_{1\le m\le M}  \sup_{0\le s <  e  \le n}\frac{\left
		    |\sum_{i=s+1}^e v_m^{\top} (X_iX_i^{\top}  -\Sigma_i)
		v_m\right |}{\sqrt {e-s}}  \le \lambda  ,  \quad   e-s \ge
	    \log(n) \right\}.
		\label{eq:B2}
	    \end{align}

\begin{lemma}\label{lemma:a1}
Suppose $\{X_i\}_{i=1}^n \subset \mathbb R^{p}$ are i.i.d  sub-Gaussian centered random vectors such that 
$$  \sup_{1\le i \le n} \|X_i \|_{\psi_2}\le B.$$
There exists an absolute constant $c > 0$ such that,
\begin{align*}
& \mathbb P (\mathcal{A}_1(\{X_i\}_{i=1}^n, B^2\sqrt{p\log(n)}))  \ge 1 - 2\times 9^p n^3  n^ {-cp},\\
& \mathbb P (\mathcal{A}_2(\{X_i\}_{i=1}^n, B^2\sqrt{p\log(n)}))  \ge 1 - 2\times 9^p n^2  n^ {-cp}, \\
& \mathbb P (\mathcal{B}_1(\{X_i\}_{i=1}^n, \{ v_m\}_{m=1}^M, B^2 \sqrt{\log (n)})) \ge 1- 2n^3Mn^{-c}, \\
&\mathbb P (\mathcal{B}_2(\{X_i\}_{i=1}^n, \{ v_m\}_{m=1}^M, B^2 \sqrt{\log (n)})) \ge 1- 2n^2Mn^{-c},
 \end{align*}
 for any set $\{v_m\}_{m=1}^M$ of deterministic unit vectors.
\end{lemma}
\begin{proof}
We first tackle $\mathcal{A}_1$.  For any $v\in\mathbb R^p$ such that
$\|v\|_2=1$, we can write
$$ v^{\top}( \tS_t  -\tSigma_t) v = \sum_{i=s+1}^e a_i  \left( (v^{\top}X_i)^2 - \mathbb{E} ((v^{\top}X_i)^2) \right) = \sum_{i=s+1}^e a_i   Z_i ,$$
where 
$$a_i = 
\begin{cases}
\sqrt {\frac{e-t}{(e-s)(t-s)}  }&  s+1\le i\le t, \\
\sqrt {\frac{t-s}{(e-s)(e-t)}} &  t+1\le i\le e,
\end{cases}
 $$
and $Z_i = (v^{\top}X_i)^2 - \mathbb{E} [(v^{\top}X_i)^2]$.  Since $\min\{t-s,
e-t\} \ge p\log (n)$, we further have that
	\[
	\sum_{i=s+1}^e a_i^2 = 1, \quad \mbox{and }\quad \max_{s+1 \leq i \leq
	e}|a_i| \le 1/\sqrt {p \log(n)} .
	\]
	Thus by Proposition 5.16 in \cite{vershynin2010introduction}, for any $\epsilon \geq 0$, 
$$ \mathbb P \left(\left | \sum_{i=s+1}^e a_iZ_ i \right|  \ge \epsilon  \right)  \le 2 \exp \left( -c\min \left\{ \frac{\epsilon^2}{K^2}, \frac{\epsilon\sqrt{p\log(n)}}{K}  \right\}\right),$$
where $c > 0$ is an absolute constant and
$$ K = \max_i \| Z_i\|_{\psi_1}  \le  2 \| (v^{\top}X_i)^2\|_{\psi_1}\le 4  \| X_i\|_{\psi_2}^2 \le 4B^2. $$
Therefore, 
\begin{align*}
\mathbb P \Bigl( \Bigl| v^{\top}\bigl( \tS_t  -\tSigma_t \bigr)v \Bigr| \ge B^2 \sqrt {p\log (n)}\Bigr) \le  2n^ {-cp}.
\end{align*}
Let $\mathcal N_{1/4}$ be a minimal $1/4$-net (with respect to the Eucidean
norm) of the unit sphere in $\mathbb
R^p$.
Then, $\mathrm{card}(\mathcal N_{1/4}) \le 9^p$ and, by a standard covering argument
followed by a union bound,
we arrive at the inequality
\begin{align*}
\mathbb P (\mathcal{A}_1(\{X_i\}_{i=1}^n, B^2\sqrt{p\log(n)}))  \ge 1 - 2\times
9^p n^3  n^ {-cp},
\end{align*}
for a universal constant $c>0$.
Following the same arguments we have,
\begin{align*}
& \mathbb P (\mathcal{A}_2(\{X_i\}_{i=1}^n, B^2\sqrt{p\log(n)}))  \ge 1 - 2\times 9^p n^2  n^ {-cp}, \\
& \mathbb P (\mathcal{B}_1(\{X_i\}_{i=1}^n, \{ v_m\}_{m=1}^M, B^2 \sqrt{\log (n)})) \ge 1- 2n^3Mn^{-c},
\end{align*}
and
\[
\mathbb P (\mathcal{B}_2(\{X_i\}_{i=1}^n, \{ v_m\}_{m=1}^M, B^2 \sqrt{\log
(n)})) \ge 1- 2n^2Mn^{-c},
\]
for some $c>0$.
\end{proof}

Let $\{s_m\}_{m=1}^M,\{e_m\}_{m=1}^M$ be two sequences independently selected at random in $[s, e]$, and   
	\begin{equation}\label{event-M}
	\mathcal{M} = \bigcap_{k = 1}^K \bigl\{s_m \in \mathcal{S}_k, e_m \in \mathcal{E}_k, \, \mbox{for some }m \in \{1, \ldots, M\}\bigr\}, 
	\end{equation}
where $\mathcal S_{k}= [\eta_k-3\Delta/4, \eta_k-\Delta/2 ]$ and $\mathcal
E_{k}= [\eta_k+\Delta/2, \eta_k+3\Delta/4 ]$, $k = 1, \ldots, K$.  In the lemma
below, we give a lower bound on the probability of $\mathcal{M}$. Under the
scaling assumed in our setting, this bound approaches $1$ as $n$ grows.

\begin{lemma}\label{lemma:random interval}
For the event $\mathcal{M}$ defined in \eqref{event-M}, we have
$$
\mathbb{P}(\mathcal M) \geq 1 -\exp\left( \log\frac{n}{\Delta}-M\frac{\Delta^2}{16 n^2} \right).
$$
\end{lemma}

\begin{proof}
Since  the number of change points are bounded by $n/\Delta $,
\begin{align*}
\mathbb{P}\bigl\{\mathcal{M}^c\bigr\} \leq \sum_{k=1}^K \prod_{m =1}^M \bigl(1 - P\bigl(s_m \in \mathcal{S}_k, e_m \in \mathcal{E}_k\bigr)\bigr)	 \leq K (1-\Delta^2/(16n^2))^M \leq n/\Delta (1 - \Delta^2/(16n^2))^M.
\end{align*}

\end{proof}

%% file: appendixb.tex

In  this section, we derive some important and useful properties of the
univariate CUSUM statistic. Our results and proofs build upon the existing
literature on univariate mean change point detection; see in partiucular,
\cite{venkatraman1992consistency} and \cite{fryzlewicz2014wild}, whose notation
will be used throughout. It is imprtnat however to note that we have made
several non-trivial modifications of those arguments, and have
made a special effort in keeping track of the changes in all the key parameters.
This careful treatment eventually allows us to achieve tight upper bounds for
the the localization rate implied by the WBSIP algorithm and which in turn
have revealed a phase transition in the problem parameters (see
\Cref{section:lb}). In particular, the results of this section can be used to
sharpen existing analyses of the BS and WBS algorithms.

\subsection{Results from  \cite{venkatraman1992consistency} }
We start by introducing some notation for one dimensional change point detection and the corresponding CUSUM statistics. Let 
 $\{Y_i\}_{i=1}^n, \{f_i\}_{i=1}^n \subset \mathbb R$ be two univariate sequences. We will make the following assumptions.
 
\begin{assumption}[Univariate mean change points]\label{assump-model 1d}
	 Let $\{\eta_k\}_{k=0}^{K+1} \subset \{0, \ldots, n\}$, where $\eta_0=0 $ and $\eta_{K+1}=n$, and 
	$$ f_{\eta_{k-1} +1} = f_{\eta_{k-1}+2} =\ldots = f_{\eta_{k}} \quad \text{for all} \quad 1\le k\le K+1,$$
	Assume
	\begin{align*}
		& \inf_{k = 1, \ldots, K+1} \{\eta_k-\eta_{k-1}\}\ge  \Delta=\Delta(n) > 0,\\
		& |f_{\eta_{k+1}} - f_{\eta_{k}} |:  = \kappa_{k} > 0, \, k = 1, \ldots, K,\\
		& \sup_{k = 1, \ldots, K+1} |f_{\eta_{k}} | <B_1. 
	\end{align*}
\end{assumption}

 We also have the corresponding CUSUM statistics over any
 generic interval $[s,e]\subset [1,T]$ defined as
	\begin{align*}
	\widetilde Y_{t}^{s,e} &=\sqrt{\frac{e-t}{(e-s) (t-s)}}\sum_{i=s+1}^{t}Y_i- \sqrt{\frac{t-s}{(e-s) (e-t)}} \sum_{i=t+1}^{e} Y_i ,\\
	\widetilde f_{t}^{s,e} &=\sqrt{\frac{e-t}{(e-s) (t-s)}}\sum_{i=s+1}^{t}f_i- \sqrt{\frac{t-s}{(e-s) (e-t)}} \sum_{i=t+1}^{e} f_i.
	\end{align*}
Throughout this \Cref{section:cusum 1d}, all of our results are proven by
regarding $\{Y_i\}_{i=1}^T$ and  $\{f_i\}_{i=1}^T$ as two deterministic sequences.
We will frequently assume that $\widetilde f_{t}^{s,e}$ is a good approximation
of $\widetilde Y_{t}^{s,e}$ in ways that we will specify through appropriate
 assumptions.

Observe that the function $\widetilde f_{t}^{s,e}$ is only well defined on
$[s,e]\cap \mathbb Z$.  Our first result,  which is taken from
\cite{venkatraman1992consistency}, shows that the there exists a continuous realization of the discrete function 
$\widetilde f_{t}^{s,e} $  
\begin{lemma}\label{lemma:continuation}
Suppose  $[s,e]\subset [1,T]$ satisfies 
 \begin{equation*}\eta_{r-1} \le s\le \eta_r \le \ldots\le \eta_{r+q} \le e \le \eta_{r+q+1}, \quad q\ge 0.
\end{equation*} 
Then there exists a continuous function $\widetilde F_{t}^{s,e} :[s,e]\to \mathbb R$
such that $\widetilde F_{r}^{s,e} = \widetilde f_{r}^{s,e}$ for every $r\in [s,e]\cap \mathbb Z$ with the following additional properties.
\\
\\
{\bf 1.} $|\widetilde F_t^{s,e}|$ is maximized at the change points within $[s,e]$.  In other words,
$$\arg\max_{s\le t\le e}  |\widetilde F_t^{s,e}| \cap \{ \eta_{r},\ldots, \eta_{r+q}\} \not = \emptyset.$$
{\bf 2.} If  $\widetilde F_t^{s,e} >0 $ for some $ t\in(s,e)$, then $\widetilde F_t^{s,e} $ is either monotonic or decreases and then increases
within each of the interval $[s,\eta_r], \ldots, [\eta_{r+q}, e]$. 
\end{lemma}
The proof of this lemma can be found in Lemmas 2.2 and 2.3 of  \cite{venkatraman1992consistency}.
We remark that if $\widetilde F_t^{s,e}\le 0 $ for all $t \in (s,e)$, then it suffices to consider the time series $\{ -f_i\}_{i=1}^T$
and a similar result as in the second part of   \Cref{lemma:continuation} still holds.
Throughout the entire section, we always view $\widetilde f^{s,e}_t $ as a continuous function and frequently invoke 
\Cref{lemma:continuation} as a basic property of the CUSUM statistics without further notice.

Our next lemma is an adaptation of a result first obtained by
\cite{venkatraman1992consistency}, which quantifies how 
fast the CUSUM statistics decays around a good change point. An analogous
result, derive using different arguments, 
can be found in Proposition 21 in \cite{wang2016high}.

\begin{lemma}[\cite{venkatraman1992consistency} Lemma 2.6]
\label{lemma:Venkatraman}
Let $[s,e]\subset [1,T]$ be any generic interval.
For some  $c_1,c_2 >0$ and $\lambda>0$ such that
\begin{align}
&\min \{\eta_k-s , e-\eta_k\} \ge c_1\Delta, \label{eq:ven 0}
\\
&\tf_{\eta_k}\ge c_2\kappa  \Delta (e-s)^{-1/2} ,    \label{eq:ven 1}
\end{align}
suppose there exists sufficiently small $c_3 > 0$ such that 
\begin{align}
\max_{s\le t\le e} |\tf_t | - \tf_{\eta_k} \le 2\lambda  \le  c_3 \kappa \Delta^3 (e-s)^{-5/2}\label{eq:ven 2}
\end{align}
 Then there exists an absolute constant $c>0$ such that if the point $d\in [s,e]$ is such that $|d-\eta_k| \le  c_1\Delta/16 $, then
\[
\widetilde f_{\eta_k}^{s,e}   -\widetilde f_{d}^{s,e}   > c \widetilde f^{s,e}_{\eta_k} |\eta_k-d | \Delta(e-s)^{-2} 
\]
\end{lemma}
\begin{remark}
If $\widetilde f^{s,e}_{\eta_k} < 0,  $ and  $d\in [s,e]$ is such that $|d-\eta_k| \le  c_1/16 $, then by considering the sequence 
$\{-f_i\}_{i=1}^n$, it holds that
$$
(-\widetilde f_{\eta_k}^{s,e})   -(-\widetilde f_{d}^{s,e} )  > c (-\widetilde f^{s,e}_{\eta_k} )|\eta_k-d | \Delta(e-s)^{-2} 
$$
\end{remark}

\begin{proof} 
 Without loss of generality, assume that $d\ge \eta_k$. 
Following the argument of  \cite{venkatraman1992consistency} Lemma~2.6, it
suffices to consider two cases: (1) $\eta_{k+1} > e$, and (2) $\eta_{k+1} \le e$.
\
\\
\\
 \textbf{Case 1.} Let $E_l$ be defined as in the   case 1 in
 \cite{venkatraman1992consistency} Lemma~2.6. There exists a $c>0$ such that,
for every $d\in [\eta_{k},\eta_k+ c_1\Delta/16]$, $\widetilde f_{\eta_k}^{s,e}
-\widetilde f_{d}^{s,e}$ (which in the notation of
\cite{venkatraman1992consistency} is the term $E_l$) can be
written as 
\[
	\tf_{\eta_k} | d-\eta_k|\frac{e-s }{\sqrt {e-\eta_k} \sqrt {\eta_k-s  +(d-\eta_k)}  \left( \sqrt { (\eta_k-s  +(d-\eta_k) ) (e-\eta_k)   } 
	+     \sqrt { (\eta_k -s ) (e-\eta_k -(d-\eta_k) )   }  \right) }.
    \]
Using the inequality $(e-s)\ge 2c_1\Delta $, the previous expression is lower
bounded by
\[
	 c'|d-\eta_k| \tf_{\eta_k}\Delta (e-s)^{-2}.
    \]
\\
\textbf{Case 2.} Let
 $h =c_1\Delta/8 $  and  $l=d-\eta_k \le h/2$. Then, following closely the initial calculations for case 2 of Lemma 2.6 of
\cite{venkatraman1992consistency}, we obtain that
$$
	\widetilde f_{\eta_k}^{s,e}   -\widetilde f_{d}^{s,e}  \ge E_{1l} (1+E_{2l}) +E_{3l},
$$
	where
	\begin{align*}
	E_{1l}&=\frac{ \tf_{\eta_k} l(h-l) } {\sqrt{(\eta_k-s+l )(e-\eta_k -l)   }   \left( \sqrt{ (\eta_k-s+l )(e-\eta_k -l)} +\sqrt{ (\eta_k-s)(e-\eta_k)} \right)}, \\
        E_{2l}&= \frac{((e-\eta_k-h )- (\eta_k-s))((e-\eta_k-h )- (\eta_k-s) -l)}{\left( \sqrt{ (\eta_k-s+l )(e-\eta_k -l)} + \sqrt{ (\eta_k-s+h)(e-\eta_k-h)} \right) } \\
        \quad & \times \frac{1}{ \left(\sqrt{ (\eta_k- s)(e-\eta_k)} +\sqrt{ (\eta_k-s+h ) (e-\eta_k-h)} \right) },\\
	\text{and} &\\
	E_{3l} &= -\frac{(\tf_{\eta_k+h}  -\tf_{\eta_k} )l }{h}  \sqrt{\frac{ (\eta_k -s+h)(e-\eta_k-h) }{  (\eta_k-s+l )(e-\eta_k-l) }}.
	\end{align*}
Since  $h= c''\Delta $ and $l\le h/2$,
	\begin{align*}
	E_{1l} \ge c''\tf_{\eta_k} | d-\eta| \Delta (e-s)^{-2}.
	\end{align*}
Observe that 
\begin{align}
\eta_k-s\le  \eta_k-s +l \le \eta_k-s +h \le 9(\eta_k-s)/8, \quad 
	e-\eta_k \ge e-\eta_k-l \ge e-\eta_k -h \ge 7(e-\eta_k)/8.\label{eq:l and h}
	\end{align}
	Thus
	\begin{align*}
	&E_{2l } 
	\\
	=&\frac{((e-\eta_k-h )- (\eta_k-s))^2 + l(h+ \eta_k-s) -l (e-\eta_k) }{\left( \sqrt{ (\eta_k-s+l )(e-\eta_k -l)} + \sqrt{ (\eta_k-s+h)(e-\eta_k-h)} \right)}\\
	\quad &  \times  \frac{1}{\left(\sqrt{ (\eta_k- s)(e-\eta_k)} +\sqrt{ (\eta_k-s+h ) (e-\eta_k-h)} \right) }
	\\
	\ge& \frac{- l(e-\eta_k) }{ (\eta_k-s+h)(e-\eta_k-h) }
	\\
        \ge& \frac{- l(e-\eta_k) }{ (\eta_k-s)(7/8)(e-\eta_k) } \ge -1/2
	\end{align*}
	where \eqref{eq:l and h} is used in the second inequality and the fact
	that $l\le h/2 \le c_1\Delta/16 \le (\eta_k-s )/16 $ is used in the last inequality.
For $E_{3l}$, observe that 
$$
	\tf_{\eta_k+h}  -\tf_{\eta_k} \le  | \tf_{\eta_k+h} | -\tf_{\eta_k} \le \max_{s\le t\le e} |   \tf_{t}  |-\tf_{\eta_k} \le 2\lambda
$$
	and that 
	 \eqref{eq:ven 0} and $l/2 \le h= c_1\Delta/8$ imply that 
	$$\eta_k-s\le  \eta_k-s +l \le \eta_k-s +h \le 9(\eta_k-s)/8 \quad
	\text{and} \quad 
	e-\eta_k \ge e-\eta_k-l \ge e-\eta_k -h \ge 7(e-\eta_k)/8.$$
	 Therefore,
	\begin{align*}
	E_{3l} &\ge - \frac{ 2(d-\eta_k) \lambda}{c_1\Delta/8}  \sqrt { \frac{(9/8)(\eta_k-s) (e- \eta_k) }{ (\eta_k-s)(7/8) (e-\eta_k )} } 
	\\
	&\ge  - \frac{ 32(d-\eta_k) \lambda}{c_1\Delta} 
	\\
	&\ge -(c''/4) \tf_{\eta_k}   (d-\eta_k) \Delta (e-s)^{-2} ,
	\end{align*}
	where the first inequality follows from \eqref{eq:l and h} and the last
	inequality follows from \eqref{eq:ven 1} and \eqref{eq:ven 2}, for a sufficiently small $c_3$.
Thus,
	$$\widetilde f_{\eta_k}^{s,e}   -\widetilde f_{d}^{s,e}  \ge E_{1l} (1+E_{2l}) +E_{3l} \ge (c''/4) \tf_{\eta_k}   | \eta_k-d| \Delta (e-s)^{-2}.$$
\end{proof}

The following proposition is a direct consequence of  \Cref{lemma:Venkatraman}
and essentially characterizes the localization
error rate of the BS algorithm.

\begin{proposition}\label{prop:1d bs}
Consider any generic interval  $(s, e) \subset (0, T)$ such that 
\begin{equation}\label{eq:location of b}
\eta_{r-1} \le s\le \eta_r \le \ldots\le \eta_{r+q} \le e \le \eta_{r+q+1}, \quad q\ge 0.
\end{equation}
Let $b \in \arg \max_{s\le t\le e}|\widetilde Y_{t}^{s,e}  | $.
Suppose   for some $c_1>0$ and $\kappa>0$, 
\begin{align}
	\max \{ \min \{ \eta_r-s ,s-\eta_{r-1}  \}, \min \{ \eta_{r+q+1}-e, e-\eta_{r+q}\} \} = \epsilon_n, \nonumber
\end{align}
where
\begin{align}
\epsilon_n  < \min \{  (3c_1/8)^2\kappa^2 \Delta^2 (e-s)^{-1}    B_1^{-2} , \Delta/4 \}, \label{prop:1d bs 4}
\end{align}
and
\begin{align}
|\widetilde Y_{b}^{s,e}  |  \ge c_1\kappa \Delta (e-s)^{-1/2}.  \label{prop:1d bs 2}
\end{align}
 Assume also that there exists sufficient small $c_3>0$ such that
 \begin{align}
&\sup_{s\le t\le e}  |\widetilde f^{s,e}_t - \widetilde Y_t^{s,e} | = \lambda_1 <
\min\{  c_3\kappa\Delta^{3}(e-s)^{-5/2}, \quad (c_1/4)\kappa  \Delta (e-s)^{-1/2}\}.
\label{prop:1d bs 1}
\end{align}
Then there exists a change point $\eta_{k} \in [s,e] $ and an absolute constant
$C_1 > 0$ such that 
\begin{align}
&\min \{e-\eta_k,\eta_k-s\}  > (3c_1/8)^2\kappa^2 \Delta^2(e-s)^{-1} B_1^{-2} \label{eq:1d re1} \\
 & |\eta_{k} -b |\le C_1\lambda_1 (e-s)^{5/2}\Delta^{-2} \kappa^{-1} , \nonumber \\
 \text{and} & \nonumber \\
 &|\tf_{\eta_k} |  \ge  |\widetilde Y_{b}^{s,e}  | -\lambda_1 \ge \max_{s\le t \le e}  |\tf_t| -2\lambda_1.
 \label{eq:1d re3}
  \end{align}
\end{proposition}
\begin{proof}
Observe that from \eqref{prop:1d bs 1},
\begin{equation}\label{eq:bound of b}
\max_{s < t < e} |\widetilde f_t^{s,e}|\le \max_{s < t < e} |\widetilde Y_{t}^{s,e} |+\lambda_1 \le| \widetilde Y_{b}^{s,e}| +\lambda_1 \le | \widetilde f_{b}^{s,e}| +2\lambda_1.
\end{equation}

Suppose $\eta_{k}\le b\le \eta_{k+1}$  for some $r-1 \le k\le  r+q $.
Observe that 
$$ |  \widetilde f_{b}^{s,e}| \ge  |  \widetilde Y_{b}^{s,e}|  -\lambda_1>  (3c_1/4)\kappa \Delta (e-s)^{-1/2} >0,$$
where the second inequality follows from \eqref{prop:1d bs 2} and \eqref{prop:1d bs 1}.
It suffices to consider the case in which  $ \widetilde f_{b}^{s,e}>0$, since,
if $ \widetilde f_{b}^{s,e}<0$, then the same arguments can be applied to the time series $\{ -f_i \}_{i=1}^n$. 
From \Cref{lemma:continuation}, $\widetilde f_t^{s,e}$ is either monotonic
 or decreasing and then increasing on $[\eta_{k},\eta_{k+1}] $. Thus
$$ \max \{\widetilde f_{\eta_{k}}^{s,e} , \widetilde f_{\eta_{k+1}}^{s,e} \}  \ge \widetilde f_{b}^{s,e}.$$
If $\widetilde f_{t}^{s,e}$ is locally decreasing at $b$, then 
$\widetilde f_{\eta_{k}}^{s,e} \ge \widetilde f_{b}^{s,e}.$
Therefore 
\begin{equation}\label{eq:decrease assumption}
\tf_{\eta_k} \ge  \widetilde f_{b}^{s,e}   >(3c_1/4)\kappa \Delta(e-s)^{-1/2}.
\end{equation}
\\
{\bf Step 1.} We first show that \eqref{eq:decrease assumption} implies
\eqref{eq:1d re1}.  For the sake of contradiction, suppose that 
$\min \{ e-\eta_k ,\eta_k-s\} \le ( (3c_1/8)\kappa \Delta(e-s)^{-1/2} B_1^{-1})^2$. 
Then
\begin{align*} \widetilde f_{\eta_k}^{s,e}   
&\le  \sqrt {\frac{(e-\eta_k) (\eta_k-s)}{e-s}  } B_1 + \sqrt {\frac{(e-\eta_k) (\eta_k-s)}{e-s}  }  B_1 \\
&\le 2 \sqrt  { \min \{ e-\eta_k ,\eta_k-s\} } B_1  \le (3c_1/4)\kappa \Delta(e-s)^{-1/2}.
\end{align*}
  This is a contradiction to \eqref{eq:decrease assumption}. Therefore  \eqref{eq:1d re1} holds for $\eta_k $.
\vskip 3mm

\noindent  {\bf Step 2.}  We now apply \Cref{lemma:Venkatraman}, since
\eqref{eq:ven 1} and \eqref{eq:ven 2} hold in virtue of \eqref{eq:decrease
assumption} and \eqref{prop:1d bs 1}, respectively.
Thus, we will need to prove that 
\[
\min \{\eta_k-s,e-\eta_k\} \ge (3/4)\Delta. 
\]
For the sake of contradiction, assume that $\min \{\eta_k-s,e-\eta_k\} < (3/4)\Delta$.
Suppose $\eta_k-s<(3/4)\Delta$. Since $\eta_{k}-\eta_{k-1}\ge\Delta $, one has $s-\eta_{k-1}\ge \Delta /4. $  This also means that 
$\eta_k$ is the first change point within $[s,e]$. Therefore $k=r$ in \eqref{eq:location of b}.
By \eqref{prop:1d bs 4}, $\min \{ \eta_r-s ,s-\eta_{r-1}  \}\le \epsilon<\Delta/4. $ Since 
$s-\eta_{k-1} =s-\eta_{p-1} \ge \Delta/4 $,   it must be the case that
$$\eta_k -s=\eta_r-s \le \epsilon \le  (3c_1/8)^2\kappa^2 \Delta^2 (e-s)^{-1}    B_1^{-2}.$$
This is a contradiction to  \eqref{eq:1d re1}. 
Therefore $\eta_k-s\ge(3/4)\Delta$.
The argument  of $e-\eta_k\ge(3/4)\Delta$ can be made analogously. 
\vskip 3mm

\noindent {\bf Step 3.} By \Cref{lemma:Venkatraman}, if there exists  a $d$ and a sufficiently large constant $C_1 > 0$ satisfying
$$d\in[\eta_{k},\eta_{k} + C_1\lambda_1 (e-s)^{5/2}\Delta^{-2} \kappa^{-1} ],$$
then
$$\widetilde f_{\eta_k}^{s,e}   -\widetilde f_{d}^{s,e}   > c \widetilde
f^{s,e}_{\eta_k} |\eta_k-d | \Delta(e-s)^{-2} \ge \lambda_1,$$
where the last inequality follows from
\eqref{eq:decrease assumption} and \eqref{prop:1d bs 1}.
\\
For the sake of contradiction, suppose $b\ge d $. Then 
$$ \widetilde f_{b}^{s,e} \le  \widetilde f_{d}^{s,e}  <  \widetilde f_{\eta_{k}}^{s,e} - \lambda_1  \le \max_{s < t < e}|\widetilde f_{t}^{s,e} | -2\lambda_1,$$
where the first inequality follows from \Cref{lemma:continuation} which ensures that  $\widetilde f_{t}^{s,e}$ is decreasing on $[\eta_{p},b] $ and $d\in [\eta_{p},b]$.
This is 
a contradiction to \eqref{eq:bound of b}. 
Thus $ b\in [\eta_{k},\eta_{k} + C_1\lambda_1 (e-s)^{5/2}\Delta^{-2} \kappa^{-1} ]$. 
\\
\eqref{eq:1d re3} follows from \eqref{eq:bound of b} and \eqref{eq:decrease assumption}.
\\
The argument for the case when $\widetilde f_b^{s,e}$ is locally increasing at
$b$ is similar and therefore we omit the details.

\end{proof}

\begin{corollary} 
\label{coro:1d bs}Let $\delta>0$ be such that $2\sqrt {\delta} B_1\le  (3c_1/4)\kappa \Delta (e-s)^{-1/2} $.
Let $b'=\arg \max_{\lceil s+ \delta       \rceil    \le t \le  \lfloor e- \delta
\rfloor}|\widetilde Y_{t}^{s,e}  | $.
Suppose all the assumption in \Cref{prop:1d bs} hold except that  \eqref{prop:1d bs 2} and
\eqref{prop:1d bs 1} are replaced by 
 \begin{align}
&|\widetilde Y_{b'}^{s,e}  |  \ge c_1\kappa \Delta (e-s)^{-1/2}  
\label{coro:1d bs 2}
\\
\text{and} & \nonumber \\
&\sup_{\lceil s+ \delta       \rceil, \ldots, \lfloor e- \delta    \rfloor}  |\widetilde f^{s,e}_t - \widetilde Y_t^{s,e} | = \lambda_1 <
\min\{  c_3\kappa\Delta^{3}(e-s)^{-5/2}, \quad (c_1/4)\kappa  \Delta (e-s)^{-1/2}\},
\label{coro:1d bs 1}
\end{align}
respectively.
Then, all the conclusions of \Cref{prop:1d bs} still hold for $b'$.
\end{corollary}
\begin{proof}
For any $t\in [s+\delta] \cup [e-\delta] $,
$$ |\widetilde f^{s,e}_t | \le 2 \sqrt {\min \{e-t, t-s\}} B_1\le 2\sqrt {\delta} B_1\le  (3c_1/4)\kappa \Delta (e-s)^{-1/2}.$$
Let $\eta_k$ and $b$ defined as in the proof of \Cref{prop:1d bs}. Then by \eqref{eq:decrease assumption}
$\eta_k,b\in \lceil s+ \delta       \rceil, \ldots, \lfloor e- \delta    \rfloor$. 
Therefore $b=b'$ and all the conclusions of  \Cref{prop:1d bs} still hold.

\end{proof}

\subsection{Results from  \cite{fryzlewicz2014wild} }
\label{subsection:fryzlewicz}
Below, we will derive some further properties of the CUSUM statistic
using the ANOVA decomposition-type of arguments first introduced by
\cite{fryzlewicz2014wild}, which are particularly effective at separating the noise from the signal. Some of the results below are contained in
\cite{fryzlewicz2014wild}, but others requires different, subtle
arguments. For completeness, we include all the proofs.  

For a pair $(s,e)$ of positive integers with $s < e$, let  $\mathcal{W}_d^{s,e}$ be the two dimensional linear subspace of
$\mathbb{R}^{e-s}$ spanned by the vectors 
	\[
	u_1 = (\underbrace{1, \ldots, 1}_{d-s}, \underbrace{0, \ldots,
	0}_{e-d})^{\top} \quad \text \quad 
	u_2 = (\underbrace{0, \ldots, 0}_{d-s}, \underbrace{1, \ldots,
	1}_{e-d})^{\top}.
	\]
For clarity, we will use $\langle \ , \ \rangle $ to denote the inner product of two vectors in the Euclidean space.

\begin{lemma}\label{lemma:WBS projection}  For $x = (x_{s+1}, \ldots, x_e)^{\top} \in \mathbb{R}^{e-s}$, let $\p^{s,e}_d(x)$ be the projection of $x$ onto $\mathcal{W}^{s,e}_d$.
\begin{enumerate}
	\item The projection $\mathcal{P}^{s,e}_d(x)$ satisfies
	\[
	\p^{s,e}_d (x) = \frac{1}{e-s}\sum_{i=s+1}^e x_i+\langle x,\psi^{s,e}_d\rangle \psi^{s,e}_d,
	\]
	where $\langle \cdot, \cdot \rangle$ is the inner product in Euclidean space, and $\psi^{s,e}_d = ( (\psi^{s,e}_d)_s, \ldots, (\psi^{s,e}_d)_{e-s})^{\top}$ with
		$$
		(\psi^{s,e}_d)_i= 
		\begin{cases}
			\sqrt\frac{e-d}{(e-s)(d-s)}, &  i = s+1, \ldots, d, \\
			-\sqrt\frac{d-s}{(e-s)(e-d)}, &  i = d+1, \ldots, e,
		\end{cases}
		$$
	i.e. the $i$-th entry of $\p^{s,e}_d (x)$ satisfies 
	$$
	\p^{s,e}_d (x)_i= \begin{cases}
		\frac{1}{d-s}\sum_{j=s+1}^d x_j, &  i = s+1, \ldots, d, \\
	\frac{1}{e-d}\sum_{j=d+1}^e x_j, & i = d+1, \ldots, e.
	\end{cases}  
	$$
	
	\item Let $\bar x =\frac{1}{e-s}\sum_{i=s+1}^e x_i$. Since $\langle \bar x,\psi^{s,e}_d \rangle=0$,
	\begin{equation} \label{eq:anova}
		\| x - \p^{s,e}_d(x) \|^2 = \| x-\bar x \|^2 - \langle x, \psi^{s,e}_d\rangle ^2.
	\end{equation}
\end{enumerate}
\end{lemma}
 
\begin{proof}
The results hold following the fact that the projection matrix of subspace $\mathcal{W}^{s,e}_d$ is
\[
P^{s,e}_{\mathcal{W}^{s,e}_d} = \left(
\begin{array}{cccccc}
	1/(d-s) & \cdots & 1/(d-s)& 0 & \cdots & 0 \\
	\vdots & \vdots & \vdots & \vdots & \vdots & \vdots \\
	1/(d-s) & \cdots & 1/(d-s) & 0 & \cdots & 0 \\
	0 & \cdots & 0 & 1/(e-d) & \cdots & 1/(e-d) \\
	\vdots & \vdots & \vdots & \vdots & \vdots & \vdots \\
	0 & \cdots & 0 & 1/(e-d)& \cdots & 1/(e-d)
\end{array}
\right).
\]
\end{proof}
For any pair $d_1, d_2 \in \{s+1, \ldots, e\}$ and $f\in \mathbb{R}^{e-s}$, the following two statements are equivalent:
\[
\langle f , \psi^{s,e}_{d_1}\rangle^2 \le   \langle  f , \psi_{d_2}^{s,e}\rangle^2  \iff  \| f -\p_{d_1}^{s,e} (f )\| ^2\ge     \| f  -\p^{s,e}_{d_2}(f )\| ^2.
\]

\begin{lemma}\label{lemma:wbs 1d}
Assume \Cref{assume:model}.  Let  $[s_0,e_0]$ be an interval with $e_0-s_0\le
C_R\Delta$ and contain at lest one change point $\eta_r$ such that 
\[
\eta_{r-1} \le s_0\le \eta_r \le \ldots\le \eta_{r+q} \le e_0 \le \eta_{r+q+1},
\quad q\ge 0.
\]
 Suppose  that $\min\{ \eta_{p'} -s_0 , e_0 -\eta_{p'} \}\ge \Delta /16$ for
 some $p'$ and let $\kse= \max\{\kappa_p: \min\{ \eta_p -s_0 , e_0 -\eta_p \} \ge \Delta /16\}$.   Consider any generic $[s,e] \subset [s_0,e_0]$, satisfying
\[
\min\{ \eta_{r} -s_0 , e_0 -\eta_{r} \} \ge \Delta /16 \quad \text{for all } \eta_r \in[s,e].
\]

Let $b \in \arg \max_{s < t < e}|\widetilde Y_{t}^{s,e}  | $.
For some $c_1>0$, $\lambda>0$ and $\delta>0$, suppose that
\begin{align}
&|\widetilde Y_{b}^{s,e}  |  \ge c_1 \kse \sqrt{\Delta},  \label{eq:wbs size of sample} \\
&\sup_{s < t < e} |\widetilde Y_{t}^{s,e}   - \widetilde f_{t}^{s,e} | \le \lambda, \nonumber \\
\text{and} & \nonumber \\
&\sup_{s_1 < t < e_1} \frac{1}{\sqrt{e_1-s_1}}\left| \sum_{t=s_1+1}^{e_1} ( Y_t
-f_t )\right| \le  \lambda  \label{eq:wbs noise 2}  \quad \text{for every} \quad
e_1-s_1 \ge \delta. 
\end{align}
If there exists a sufficiently small $c_2 > 0$ such that
\begin{equation}\label{eq:wbs noise}
\lambda\le c_2\kse\sqrt \Delta \quad \text{and} \quad  \delta \le c_2\Delta,
\end{equation}
then there exists a change point $\eta_{k} \in (s, e)$  such that 
\[
    \min \{e-\eta_k,\eta_k-s\}  >  \Delta /4  \quad \text{and} \quad 
|\eta_{k} -b |\le \min \{C_3\lambda^2\kappa_k^{-2},\delta\}.
\]
\end{lemma}

\begin{proof}Without loss of generality, assume that $\widetilde f_{b}^{s,e}>0$ and that $\widetilde f_t^{s,e} $ is locally decreasing at $b$.
Observe that there has to be a change point $\eta_k \in [s,b]$, or otherwise $\widetilde f_b^{s,e} >0 $  implies that  $\widetilde f_t^{s,e} $ is decreasing,
as a consequence of  \Cref{lemma:cusum boundary bound}.

Thus if $s\le \eta_k\le b \le  e $, then 
\begin{align}\widetilde f_{\eta_k}^{s,e}\ge \widetilde f_{b}^{s,e} \ge |\widetilde Y^{s,e}_b  | -\lambda  \ge c_1 \kse \sqrt{\Delta} -c_2 \kse \sqrt \Delta 
\ge  (c_1/2) \kse \sqrt {\Delta}. \label{eq:wbs size of change point}
\end{align}
Observe that $e-s\le e_0-s_0\le C_R\Delta $ and that $(s, e)$ has to contain at least one change point or otherwise $ |\widetilde f^{s,e}_{\eta_k} |  =0 $ which contradicts  \eqref{eq:wbs size of change point}.
\vskip 3mm
\noindent {\bf Step 1.} 
In this step, we are to show that $\min\{ \eta_k -s , e -\eta_k \} \ge \min\{1, c_1^2 \}\Delta /16$. 

Suppose $\eta_k$ is the only change point in $(s, e)$. 
So $\min\{ \eta_k -s , e -\eta_k \} \ge   \min\{1, c_1^2\}\Delta /16$ must hold or otherwise it follows from \Cref{lemma:one change point basics}, we have
	\[
	    |\widetilde f^{s,e}_{\eta_k} | < \frac{c_1 }{4} \kappa_k
	    \sqrt{\Delta} \le \frac{c_1}{2}  \kse \sqrt{\Delta},
	\]
	which contradicts \eqref{eq:wbs size of change point}.

Suppose $(s, e)$ contains at least two change points. Then $ \eta_k -s \le   \min\{1,c_1^2  \}\Delta /16 $ implies that
$\eta_k$ is the first change point in $[s,e]$. 
Therefore 
\begin{align*}
 |\widetilde f^{s,e}_{\eta_k}| \le \frac{1}{4}  |  \widetilde f^{s,e}_{\eta_{k+1}}| +2\kappa_r  \sqrt {\eta_r -s} 
 \le \frac{1}{4}\max_{s < t < e}|  \widetilde f^{s,e}_{t}|+\frac{c_1}{2}\kappa_r  \sqrt {\Delta}
 \\
 \le \frac{1}{4}|  \widetilde Y^{s,e}_{b}| +\lambda +\frac{c_1}{2}\kse  \sqrt {\Delta}  \le\frac{3}{4}|  \widetilde Y^{s,e}_{b}| +\lambda< |  \widetilde Y^{s,e}_{b}|  -\lambda
 \end{align*}
where the first inequality follows from \Cref{lemma:cusum boundary bound}, the fourth inequality follows from \eqref{eq:wbs size of sample}, and the last inequality holds when $c_2$ is sufficiently small.  This contradicts \eqref{eq:wbs size of change point}.

\vskip 3mm
\noindent{\bf Step 2.} By \Cref{lemma:Venkatraman} there exists $d$ such that  
$$d\in[\eta_{k},\eta_{k} + \lambda \sqrt \Delta (\kse)^{-1} ]$$
and that 
 $
  \widetilde f_{\eta_{k}}^{s,e}   -\widetilde f_{d}^{s,e}   > 2\lambda.
$
For the sake of contradiction, suppose $b\ge d $. Then 
$$ \widetilde f_{b}^{s,e} \le  \widetilde f_{d}^{s,e}  <  \widetilde f_{\eta_{k}}^{s,e} -2\lambda  \le \max_{s < t < e}|\widetilde f_{t}^{s,e} | -2\lambda
\le \max_{s < t < e} |\widetilde Y^{s,e}_t| +\lambda-2 \lambda  = |\widetilde Y^{s,e}_b| -\lambda ,$$
where the first inequality follows from \Cref{lemma:continuation}, which ensures that  $\widetilde f_{t}^{s,e}$ is decreasing on $[\eta_{p},b] $ and $d\in [\eta_{p},b]$.
This is 
a contradiction to \eqref{eq:wbs size of change point}. 
Thus $ b\in [\eta_{k},\eta_{k} + \lambda \sqrt \Delta (\kse)^{-1} ]$. 

\vskip 3mm
\noindent {\bf Step 3.}
 Let $f^{s,e} =(f_{s+1},\ldots, f_e)^{\top} \in \mathbb{R}^{(e-s)} $ and $Y^{s,e}=(Y_{s+1}, \ldots, Y_e)^{\top} \in \mathbb{R}^{(e-s)}$. 
  By the definition of $b$, it holds that
	$$
	\bigl\|Y^{s,e} - \mathcal{P}^{s,e}_{b}(Y^{s,e})\bigr\|^2 \leq \bigl \|Y^{s,e} - \mathcal{P}^{s,e}_{\eta_k}(Y^{s,e})\bigr\|^2 
	\leq \bigl\|Y^{s,e} - \mathcal{P}_{\eta_k}^{s,e}(f^{s,e})\bigr\|^2.
	$$
For the sake of contradiction, throughout the rest of this argument suppose
that, for some sufficiently large constant $C_3 > 0$ to be specified,
\begin{align}\label{eq:wbs contradict assume}
\eta_k + \max\{C_3\lambda^2\kappa_k^{-2},\delta\}< b . \end{align}
(This will of course imply that $\eta_k + \max\{C_3\lambda^2
    (\kse)^{-2},\delta\}< b$).
	We will show that this leads to the bound
	\begin{align}\label{eq:WBS sufficient}
	\bigl\|Y^{s,e} - \mathcal{P}_{b}^{s,e} (Y^{s,e})\bigr\|^2 >
	\bigl\|Y^{s,e} - \mathcal{P}^{s,e}_{\eta_k}(f^{s,e})\bigr\|^2,	
    \end{align}
which is a contradiction. 
	To derive \eqref{eq:WBS sufficient} from \eqref{eq:wbs contradict assume}, we note that  $\min\{ e-\eta_k,\eta_k-s\}\ge  \min\{1, c_1^2 \}\Delta/16$ and that 
	$| b- \eta_k| \le  \lambda \sqrt \Delta (\kse)^{-1}$ implies that 
\begin{align}
\label{eq:wbs size of intervals}
 \min\{ e-b, b-s\} \ge   \min\{1, c_1^2 \}\Delta /16 -\lambda \sqrt \Delta (\kse)^{-1} \ge \min\{1, c_1^2 \}\Delta /32 ,
 \end{align}
where the last inequality follows from \eqref{eq:wbs noise} and holds for an
appropriately small $c_2>0$.

\Cref{eq:WBS sufficient} is in turn implied by 
\begin{equation}
\label{eq:WBS sufficient 2}
2\langle \varepsilon^{s,e} ,\p_b(Y^{s,e}) - \p _{\eta_k}(f^{(s,e)})\rangle < \|f^{s,e}-\p_b(f^{s,e}) \|^2 -\|f^{s,e}-\p_{\eta_k}(f^{s,e}) \|^2,
 \end{equation}
 where $\varepsilon^{s,e}= Y^{s,e}-f^{s,e}$.
By \eqref{eq:anova}, the right hand side of \eqref{eq:WBS sufficient 2}
satisfied the relationships 
\begin{align*}
  \|f^{s,e}-\p_b(f^{s,e}) \|^2 -\|f^{s,e}-\p_{\eta_k}(f^{s,e}) \|^2
& =
 \langle f^{s,e} , \psi_{\eta_k} \rangle^2 -\langle f^{s,e} , \psi_{b} \rangle^2  \\
&=  (\tf_{\eta_p})^2 -(\tf_{b})^2\\
& \ge ( \tf_{\eta_k} - \tf_{b} ) | \tf_{\eta_k}|\\
& \ge  c |d-\eta_k|  (\tf_{\eta_k})^2 \Delta^{-1}\\
& \ge c' |d-\eta_k |(\kse)^2,
\end{align*}
where \Cref{lemma:Venkatraman} and 
\eqref{eq:wbs size of change point} are used in the second  and third
inequalities.
The left hand side of \eqref{eq:WBS sufficient 2}
can  in turn be rewritten as  
\begin{equation}
\label{eq:perturbations}
2\langle \varepsilon^{s,e} ,\p_b(X^{s,e}) - \p _{\eta_k}(f^{s,e})\rangle = 2 \langle \varepsilon^{s,e}, \p_b(X^{s,e})-\p_b(f^{s,e})\rangle + 
2  \langle \varepsilon^{s,e}, \p_b(f^{s,e})-\p_{\eta_k} (f^{s,e})\rangle .
 \end{equation}
The second term on the right hand side  of the previous display can be
decomposed as
\begin{align*}
\langle \varepsilon^{s,e}  , \p_b (f ^{s,e})-\p_{\eta_k} (f^{s,e} )\rangle &  =  \left( \sum_{i=s+1}^{\eta_k} +\sum_{i={\eta_k}+1}^b +\sum_{i=b+1}^e\right) 
 \varepsilon^{s,e}_i \left( \p_b(f^{s,e})_i  - \p_{\eta_k}(f^{s,e})_i  \right)\\
 &= I +II +III.
\end{align*}
In order to bound the terms $I$, $II$ and $III$, observe that,
 since $e-s\le e_0-s_0\le  C_R\Delta$, the interval $[s,e]$ must contain at most $C_R+1$ change points.
Let 
$$ \eta_{r'-1}< s\le \eta_{r'} \le \ldots  \le \eta_{p'+q'}< e\le \eta_{p'+q'+1}. $$
Then $p'+q'+1-r'\le C_R +1$. 
\vskip 3mm
\noindent {\bf Step 4.}
We can write 
\begin{align*} I
=& \sqrt{{\eta_k} -s}\left (\frac{1}{\sqrt{{\eta_k} -s}} \sum_{i=s+1}^{\eta_k} \varepsilon^{s,e}_i\right)
 \left( \frac{1}{b-s} \sum_{i=s+1}^b f_i-\frac{1}{{\eta_k}-s} \sum_{i=s+1}^{\eta_k} f_i\right).  \\
  \end{align*}
  Thus,
 \begin{align*} 
& \left |\frac{1}{b-s} \sum_{i=s+1}^b f_i  -\frac{1}{{\eta_k}-s} \sum_{i=s+1}^{\eta_k} f_i \right | = \left|  \frac{ (\eta_k -s ) (\sum_{i=s+1}^{\eta_k}  f_i +\sum_{i=\eta_k+1}^{b}  f_i)  - (b-s) \sum_{i=s+1}^{\eta_k} f_i }{(b-s)(\eta_k -s)}   \right|\\
=&
  \left|  \frac{ (\eta_k -b ) \sum_{i=s+1}^{\eta_k}  f_i + (\eta_k-s)\sum_{i=\eta_k+1}^{b}  f_i)  }{(b-s)(\eta_k -s)}   \right| =   \left|  \frac{ (\eta_k -b ) \sum_{i=s+1}^{\eta_k}  f_i + (\eta_k-s) (b-\eta_k) f_{\eta_k+1})  }{(b-s)(\eta_k -s)}   \right|\\
   =&
  \frac{b-\eta_k}{b-s}\left|  - \frac{1}{\eta_k -s } \sum_{i=s+1}^{\eta_k}  f_i+ f_{\eta_{k+1}}  \right| \le \frac{b-\eta_k}{b-s}  (C_R +1 )\kse
   \end{align*}
where \Cref{lemma:number of changes} is used in the last inequality.  It follows from \Cref{eq:wbs noise 2} that
\begin{align*}
| I|\le  \sqrt{\eta_k-s}\lambda  \frac{|b-\eta_k|}{b-s}(C_R+1) \kse \le \frac{4\sqrt{2}}{\min\{1, c_1\}}|b-\eta_k | \Delta^{-1/2} \lambda  (C_R+1) \kse,
\end{align*}
where \eqref{eq:wbs size of intervals} is used in the last inequality.
\vskip 3mm
\noindent {\bf Step 5.}
For the second term $II$, we have that 
\begin{align*} |II|
=&\left| \sqrt{{b-\eta_k} }\left (\frac{1}{\sqrt{{b-\eta_k}}} \sum_{i={\eta_k+1}}^d \varepsilon^{s,e}_i\right)
 \left (\frac{1}{b-s} \sum_{i=s+1}^b f_i-\frac{1}{e-{\eta_k}} \sum_{i={\eta_{k}+1 } }^e f_i\right)  \right|\\
\le&\sqrt {b-{\eta_k}}    \lambda \left( \left| f_{\eta_k} -f_{\eta_{k+1} }\right| + 
\left|\frac{1}{b-s} \sum_{i=s+1 }^b f_i- f_{\eta_k}  \right| + \left |\frac{1}{e-{\eta_k}} \sum_{i={\eta_k+1}}^e f_i - f_{\eta_{k+1} } \right| \right)\\
\le & \sqrt {b-{\eta_k}} (  \kse + (C_R +1)\kse +(C_R+1)\kse),
 \end{align*}
 where the first inequality follows from \eqref{eq:wbs size of intervals} and \eqref{eq:wbs noise 2},
 and the second inequality from \Cref{lemma:number of changes}.
\vskip 3mm

\noindent {\bf Step 6.}
Finally, we have that
\begin{align*}
III = \sqrt{e-b}\left(\frac{1}{e-b}\sum_{i = b+1}^e \varepsilon^{s, e}_i\right)\left(\frac{1}{e-\eta_k}\sum_{i=\eta_k + 1}^e f_i - \frac{1}{e-b}\sum_{i=b+1}^e f_i\right).	
\end{align*}
Therefore,
\begin{align*}
|III| \leq \sqrt{e-b}\lambda \frac{b-\eta_k}{e-b}(C_R+1)\kse	 \leq \frac{4\sqrt{2}}{\min\{1, c_1\}}|b-\eta_k | \Delta^{-1/2} \lambda  (C_R+1) \kse.
\end{align*}

\vskip 3mm

\noindent {\bf Step 7.} 
Using the first part of \Cref{lemma:WBS projection}, the first term on the right hand side of \eqref{eq:perturbations} can be bounded  as
$$\langle \varepsilon^{s,e}, \p_d(X^{s,e})-\p_d(f^{s,e})\rangle\le \lambda^2.$$
 Thus \eqref{eq:WBS sufficient 2}
 holds if 
\[
|b-\eta_k|(\kse)^2 \ge  C  \max\left\{  |b-\eta_k| \Delta^{-1/2} \lambda \kse  ,\quad  \sqrt {b-{\eta_k}}    \lambda \kse  ,\quad \lambda^2 \right\}.
\]
Since $\lambda \le c_3\sqrt \Delta \kappa  $, the first inequality holds. The second inequality follows from
$|b-\eta_k| \ge C_3\lambda^2 (\kappa_k)^{-2} \ge C_3\lambda^2 (\kse)^{-2}$, as
assumed in \eqref{eq:wbs contradict assume}.
This completes the proof.
\end{proof}

\begin{corollary}
\label{coro:wbs 1d} Let  $[s_0,e_0]$ be a generic interval satisfying $e_0-s_0\le
C_R\Delta$ and containing at lest one change point $\eta_r$ such that 
\[
\eta_{r-1} \le s_0\le \eta_r \le \ldots\le \eta_{r+q} \le e_0 \le \eta_{r+q+1},
\quad q\ge 0.
\]
 Suppose $\min\{ \eta_{p'} -s_0 , e_0 -\eta_{p'} \}\ge \Delta /16$ for some $p'$
 and, let 
$\kse= \max\{\kappa_p: \min\{ \eta_p -s_0 , e_0 -\eta_p \} \ge \Delta /16\}$.  Consider a generic interval $(s, e) \subset (s_0, e_0)$, satisfying  
\[
\min\{ \eta_{p} -s_0 , e_0 -\eta_{p} \} \ge \Delta /16 \quad \text{for all } \eta_p \in[s,e].
\]
Let $\delta'>0$ be some constant and  $b'=\arg \max_{\lceil s+ \delta '      \rceil, \ldots, \lfloor e- \delta '   \rfloor}|\widetilde Y_{t}^{s,e}  | $.
Suppose  in addition that, for some positive constants $c_1$ and $c_2$,  
\begin{align}
 \label{eq:coro wbs size of sample}
&|\widetilde Y_{b'}^{s,e}  |  \ge c_1 \kse \sqrt{\Delta}  ,
\\
&\sup_{t= \lceil s+ \delta'       \rceil, \ldots, \lfloor e- \delta '   \rfloor}
|\widetilde Y_{t}^{s,e}   - \widetilde f_{t}^{s,e} | \le \lambda, \label{eq:coro wbs noise 1}
\\
&\sup_{s_1\le t\le e_1} \frac{1}{\sqrt{e_1-s_1}}| \sum_{t=s_1+1}^{e_1} ( Y_t
-f_t )| \le  \lambda , \label{eq:coro wbs noise 2}
\quad \text{for every} \quad e_1-s_1 \ge \delta',
\\
&\lambda\le c_2\kse\sqrt \Delta, \label{eq:coro wbs noise}\\
\text{and} & \nonumber \\ 
&\delta'   \le    c_2 \Delta \label{eq:coro wbs short spacing}.
\end{align} 
 Then there exists a change point $\eta_{k} \in [s,e] $  such that 
\begin{align*}
&\min \{e-\eta_k,\eta_k-s\}  >  \Delta /4 \\
&|\eta_{k} -b' |\le \max\{ C_3\lambda^2\kappa_k^{-2},\delta'\}.
\end{align*}
\end{corollary}
\begin{proof}
By the same proof of \Cref{coro:1d bs}, if $b$ is defined as in \Cref{lemma:wbs 1d}, then $b=b'$ if $c_2$ is sufficiently small.
\end{proof}

%% file: draft.bbl
\begin{thebibliography}{50}
\expandafter\ifx\csname natexlab\endcsname\relax\def\natexlab#1{#1}\fi
\expandafter\ifx\csname url\endcsname\relax
  \def\url#1{\texttt{#1}}\fi
\expandafter\ifx\csname urlprefix\endcsname\relax\def\urlprefix{URL }\fi
\providecommand{\eprint}[2][]{\url{#2}}

\bibitem[{Anderson(2003)}]{Anderson2003}
\textsc{Anderson, T.~W.} (2003).
\newblock \textit{An Introduction to Multivariate Statistical Analysis}.
\newblock 3rd ed. Wiley Series in Probability and Statistics, John Wiley \&
  Sons.

\bibitem[{Aston and Kirch(2014)}]{AstonKirch2014}
\textsc{Aston, J. A.~D.} and \textsc{Kirch, C.} (2014).
\newblock Efficiency of change point tests in high dimensional settings.
\newblock \textit{arXiv preprint arXiv: 1409.1771}.

\bibitem[{Aue et~al.(2009)Aue, H\"{o}mann, Horv\'{a}th and
  Reimherr}]{AueEtal2009}
\textsc{Aue, A.}, \textsc{H\"{o}mann, S.}, \textsc{Horv\'{a}th, L.} and
  \textsc{Reimherr, M.} (2009).
\newblock Break detection in the covariance structure of multivariate nonlinear
  time series models.
\newblock \textit{The Annals of Statistics}, \textbf{37} 4046--4087.

\bibitem[{Avanesov and Buzun(2016)}]{avanesov2016change}
\textsc{Avanesov, V.} and \textsc{Buzun, N.} (2016).
\newblock Change-point detection in high-dimensional covariance structure.
\newblock \textit{arXiv preprint arXiv:1610.03783}.

\bibitem[{Baranowski et~al.(2016)Baranowski, Chen and
  Fryzlewicz}]{BaranowskiEtal2016}
\textsc{Baranowski, R.}, \textsc{Chen, Y.} and \textsc{Fryzlewicz, P.} (2016).
\newblock Narrowest-{O}ver-{T}hreshold detection of multiple change-points and
  change-point-like feature.
\newblock \textit{arXiv preprint arXiv: 1609.00293}.

\bibitem[{Barigozzi et~al.(2016)Barigozzi, Cho and
  Fryzlewicz}]{BarigozziEtal2016}
\textsc{Barigozzi, M.}, \textsc{Cho, H.} and \textsc{Fryzlewicz, P.} (2016).
\newblock Simultaneous multiple change-point and factor analysis for
  high-dimensional time series.
\newblock \textit{arXiv preprint arXiv: 1612.06928}.

\bibitem[{Berkes et~al.(2009)Berkes, Gabrys, Horv\'{a}th and
  Kokoszka}]{BerkesEtal2009}
\textsc{Berkes, I.}, \textsc{Gabrys, R.}, \textsc{Horv\'{a}th, L.} and
  \textsc{Kokoszka, P.} (2009).
\newblock Detecting changes in the mean of functional observations.
\newblock \textit{Journal of the Royal Statistical Society: Series B
  (Statistical Methodology)}, \textbf{71} 927--946.

\bibitem[{Berthet and Rigollet(2013)}]{berthet2013optimal}
\textsc{Berthet, Q.} and \textsc{Rigollet, P.} (2013).
\newblock Optimal detection of sparse principal components in high dimension.
\newblock \textit{The Annals of Statistics}, \textbf{41} 1780--1815.

\bibitem[{Birke and Dette(2005)}]{BirkeDette2005}
\textsc{Birke, M.} and \textsc{Dette, H.} (2005).
\newblock A note on testing the covariance matrix for large dimension.
\newblock \textit{Statistics and Probability Letters}, \textbf{74} 281--289.

\bibitem[{Cai and Ma(2013)}]{cai2013optimal}
\textsc{Cai, T.~T.} and \textsc{Ma, Z.} (2013).
\newblock Optimal hypothesis testing for high dimensional covariance matrices.
\newblock \textit{Bernoulli}, \textbf{19} 2359--2388.

\bibitem[{Chan and Walther(2013)}]{chan2009}
\textsc{Chan, H.~P.} and \textsc{Walther, G.} (2013).
\newblock Detection with the scan and the average likelihood ratio.
\newblock \textit{Statistica Sinica}, \textbf{1} 409--428.

\bibitem[{Cho(2015)}]{Cho2015}
\textsc{Cho, H.} (2015).
\newblock Change-point detection in panel data via double cusum statistic.
\newblock \textit{Electronic Journal of Statistics} in press.

\bibitem[{Cho and Fryzlewicz(2012)}]{ChoFryzlewicz2012}
\textsc{Cho, H.} and \textsc{Fryzlewicz, P.} (2012).
\newblock Multiscale and multilevel technique for consistent segmentation of
  nonstationary time series.
\newblock \textit{Statistica Sinica}, \textbf{22} 207--229.

\bibitem[{Cho and Fryzlewicz(2015)}]{ChoFryzlewicz2015}
\textsc{Cho, H.} and \textsc{Fryzlewicz, P.} (2015).
\newblock Multiple change-point detection for high-dimensional time series via
  {S}parsified {B}inary {S}egmentation.
\newblock \textit{Journal of the Royal Statistical Society: Series B
  (Statistical Methodology)}, \textbf{77} 475--507.

\bibitem[{Davies and Kovac(2001)}]{davies2001}
\textsc{Davies, P.~L.} and \textsc{Kovac, A.} (2001).
\newblock Local extremes, runs, strings and multiresolution.
\newblock \textit{Ann. Statist.}, \textbf{29} 1--65.

\bibitem[{Davis et~al.(2006)Davis, Lee and Rodriguez-Yam}]{DavisEtal2006}
\textsc{Davis, R.~A.}, \textsc{Lee, T. C.~M.} and \textsc{Rodriguez-Yam, G.~A.}
  (2006).
\newblock Structural break estimation for nonstationary time series models.
\newblock \textit{Journal of the American Statistical Association},
  \textbf{101} 223--239.

\bibitem[{Frick et~al.(2014)Frick, Munk and Sieling}]{FrickEtal2014}
\textsc{Frick, K.}, \textsc{Munk, A.} and \textsc{Sieling, H.} (2014).
\newblock Multiscale change point inference.
\newblock \textit{Journal of the Royal Statistical Society: Series B
  (Statistical Methodology)}, \textbf{76} 495--580.

\bibitem[{Fryzlewicz(2014)}]{fryzlewicz2014wild}
\textsc{Fryzlewicz, P.} (2014).
\newblock Wild binary segmentation for multiple change-point detection.
\newblock \textit{The Annals of Statistics}, \textbf{42} 2243--2281.

\bibitem[{Gombay et~al.(1996)Gombay, Horv\'{a}th and
  Hu\v{s}kov\'{a}}]{GombayEtal1996}
\textsc{Gombay, E.}, \textsc{Horv\'{a}th, L.} and \textsc{Hu\v{s}kov\'{a}, M.}
  (1996).
\newblock Estimators and tests for change in variances.
\newblock \textit{Statistics and Risk Modeling}, \textbf{14} 145--160.

\bibitem[{Harchaoui and L\'{e}vy-Leduc(2010)}]{HarchaouiLevy2010}
\textsc{Harchaoui, Z.} and \textsc{L\'{e}vy-Leduc, C.} (2010).
\newblock Multiple change-point estimation with a total variation penalty.
\newblock \textit{Journal of American Statistical Association}, \textbf{105}
  1480--1493.

\bibitem[{Harchaoui and L{\'e}vy-Leduc(2010)}]{Harchoui}
\textsc{Harchaoui, Z.} and \textsc{L{\'e}vy-Leduc, C.} (2010).
\newblock Multiple change-point estimation with a total variation penalty.
\newblock \textit{Journal of the American Statistical Association},
  \textbf{105} 1480--1493.

\bibitem[{Horv\'{a}th and Hu\v{s}kov\'{a}(2012)}]{HorvathHuskova2012}
\textsc{Horv\'{a}th, L.} and \textsc{Hu\v{s}kov\'{a}, M.} (2012).
\newblock Change-point detection in panel data.
\newblock \textit{Journal of Time Series Analysis}, \textbf{33} 631--648.

\bibitem[{Inclan and Tiao(1994)}]{InclanTiao1994}
\textsc{Inclan, C.} and \textsc{Tiao, G.~C.} (1994).
\newblock Use of cumulative sums of squares for retrospective detection of
  changes of variance.
\newblock \textit{Journal of the American Statistical Association}, \textbf{89}
  913--923.

\bibitem[{James et~al.(1987)James, James and Siegmund}]{JamesEtal1987}
\textsc{James, B.}, \textsc{James, K.~L.} and \textsc{Siegmund, D.} (1987).
\newblock Tests for a change-point.
\newblock \textit{Biometrika}, \textbf{74} 71--83.

\bibitem[{Jirak(2015)}]{Jirak2015}
\textsc{Jirak, M.} (2015).
\newblock Uniform change point tests in high dimension.
\newblock \textit{The Annals of Statistics}, \textbf{43} 2451--2483.

\bibitem[{Johnstone(2001)}]{Johnstone2011}
\textsc{Johnstone, I.~M.} (2001).
\newblock On the distribution of the largest eigenvalue in principal components
  analysis.
\newblock \textit{The Annals of Statistics}, \textbf{29} 295--327.

\bibitem[{Killick et~al.(2012)Killick, Fearnhead and Eckley}]{KillickEtal2012}
\textsc{Killick, R.}, \textsc{Fearnhead, P.} and \textsc{Eckley, I.~A.} (2012).
\newblock Optimal detection of changepoints with a linear computational cost.
\newblock \textit{Journal of the American Statistical Association},
  \textbf{107} 1590--1598.

\bibitem[{Korkas and Fryzlewicz(2017)}]{korkas2017multiple}
\textsc{Korkas, K.~K.} and \textsc{Fryzlewicz, P.} (2017).
\newblock Multiple change-point detection for non-stationary time series using
  wild binary segmentation.
\newblock \textit{Statistica Sinica}, \textbf{27} 287--311.

\bibitem[{Lavielle(1999)}]{Lavielle1999}
\textsc{Lavielle, M.} (1999).
\newblock Detection of multiple changes in a sequence of dependent variables.
\newblock \textit{Stochastic Processes and their Applications}, \textbf{83}
  79--102.

\bibitem[{Lavielle and Moulines(2000)}]{LavielleMoulines2000}
\textsc{Lavielle, M.} and \textsc{Moulines, E.} (2000).
\newblock Least-squares estimation of an unknown number of shifts in a time
  series.
\newblock \textit{Journal of Time Series Analysis}, \textbf{21} 33--59.

\bibitem[{Ledoit and Wolf(2002)}]{LedoitWolf2002}
\textsc{Ledoit, O.} and \textsc{Wolf, M.} (2002).
\newblock Some hypothesis tests for the covariance in principal components
  analysis.
\newblock \textit{The Annals of Statistics}, \textbf{30} 1081--1102.

\bibitem[{Li et~al.(2017)Li, Guo and Munk}]{LiEtal2017}
\textsc{Li, H.}, \textsc{Guo, Q.} and \textsc{Munk, A.} (2017).
\newblock Multiscale change-point segmentation: Beyond step functions.
\newblock \textit{arXiv preprint arXiv: 1708.03942}.

\bibitem[{Lin et~al.(2017)Lin, Sharpnack, Rinaldo and Tibshirani}]{KevinNIPS}
\textsc{Lin, K.}, \textsc{Sharpnack, J.~L.}, \textsc{Rinaldo, A.} and
  \textsc{Tibshirani, R.~J.} (2017).
\newblock A sharp error analysis for the fused lasso, with application to
  approximate changepoint screening.
\newblock In \textit{Advances in Neural Information Processing Systems 30}
  (I.~Guyon, U.~V. Luxburg, S.~Bengio, H.~Wallach, R.~Fergus, S.~Vishwanathan
  and R.~Garnett, eds.). 6887--6896.

\bibitem[{Olshen et~al.(2004)Olshen, Venkatraman, Lucito and
  Wigler}]{olshen2004circular}
\textsc{Olshen, A.~B.}, \textsc{Venkatraman, E.}, \textsc{Lucito, R.} and
  \textsc{Wigler, M.} (2004).
\newblock Circular binary segmentation for the analysis of array-based dna copy
  number data.
\newblock \textit{Biostatistics}, \textbf{5} 557--572.

\bibitem[{Page(1954)}]{Page1954}
\textsc{Page, E.~S.} (1954).
\newblock Continuous inspection schemes.
\newblock \textit{Biometrika}, \textbf{41} 100--115.

\bibitem[{Picard(1985)}]{Picard1985}
\textsc{Picard, D.} (1985).
\newblock Testing and estimating change-points in time series.
\newblock \textit{Advances in Applied Probability}, \textbf{17} 841--867.

\bibitem[{Qian and Jia(2012)}]{Qian_Jia}
\textsc{Qian, J.} and \textsc{Jia, J.} (2012).
\newblock On pattern recovery of the fused lasso.
\newblock Tech. rep.
\newblock Https://arxiv.org/abs/1211.5194.

\bibitem[{Rojas and Wahlberg(2014)}]{Rojas}
\textsc{Rojas, C.~R.} and \textsc{Wahlberg, B.} (2014).
\newblock On change point detection using the fused lasso method.
\newblock Tech. rep.
\newblock Https://arxiv.org/abs/1401.5408.

\bibitem[{Roy(1957)}]{Roy1957}
\textsc{Roy, S.~N.} (1957).
\newblock \textit{Some Aspects of Multivariate Analysis}.
\newblock New York: Wiley.

\bibitem[{Siegmund and Venkatraman(1995)}]{SiegmundVenkatraman1995}
\textsc{Siegmund, D.} and \textsc{Venkatraman, E.~S.} (1995).
\newblock Using the generalized likelihood ratio statistic for sequential
  detection of a change-point.
\newblock \textit{The Annals of Statistics} 255--271.

\bibitem[{Tibshirani et~al.(2005)Tibshirani, Saunders, Rosset, Zhu and
  Knight}]{TibshiraniEtal2005}
\textsc{Tibshirani, R.}, \textsc{Saunders, M.}, \textsc{Rosset, S.},
  \textsc{Zhu, J.} and \textsc{Knight, K.} (2005).
\newblock Sparsity and smoothness via the fused lasso.
\newblock \textit{Journal of the Royal Statistical Society: Series B
  (Statistical Methodology)}, \textbf{67} 91--108.

\bibitem[{Venkatraman(1992)}]{venkatraman1992consistency}
\textsc{Venkatraman, E.~S.} (1992).
\newblock \textit{Consistency results in multiple change-point problems}.
\newblock Ph.D. thesis, Stanford University.

\bibitem[{Vershynin(2010)}]{vershynin2010introduction}
\textsc{Vershynin, R.} (2010).
\newblock Introduction to the non-asymptotic analysis of random matrices.
\newblock \textit{arXiv preprint arXiv:1011.3027}.

\bibitem[{Vostrikova(1981)}]{vostrikova1981detection}
\textsc{Vostrikova, L.} (1981).
\newblock Detection of the disorder in multidimensional random-processes.
\newblock \textit{Doklady Akademii Nauk SSSR}, \textbf{259} 270--274.

\bibitem[{Wald(1945)}]{Wald1945}
\textsc{Wald, A.} (1945).
\newblock Sequential tests of statistical hypotheses.
\newblock \textit{The Annals of Mathematical Statistics}, \textbf{16} 117--186.

\bibitem[{Wang and Samworth(2016)}]{wang2016high}
\textsc{Wang, T.} and \textsc{Samworth, R.~J.} (2016).
\newblock High-dimensional changepoint estimation via sparse projection.
\newblock \textit{arXiv preprint arXiv:1606.06246}.

\bibitem[{Wang(1995)}]{Wang1995}
\textsc{Wang, Y.} (1995).
\newblock Jump and sharp cusp detection by wavelets.
\newblock \textit{Biometrika}, \textbf{82} 385--397.

\bibitem[{Wu(2005)}]{Wu2005}
\textsc{Wu, W.~B.} (2005).
\newblock Nonlinear system theory: {A}nother look at dependence.
\newblock \textit{Proceedings of the National Academy of Sciences of the United
  States of America}, \textbf{102} 14150--14154.

\bibitem[{Yao and Au(1989)}]{YaoAu1989}
\textsc{Yao, Y.-C.} and \textsc{Au, S.-T.} (1989).
\newblock Least-squares estimation of a stop function.
\newblock \textit{Sankhy\={a}: The Indian Journal of Statistics, Series A}
  370--381.

\bibitem[{Yu(1997)}]{yu1997assouad}
\textsc{Yu, B.} (1997).
\newblock \textit{Festschrift for Lucien Le Cam}, vol. 423, chap. Assouad,
  {F}ano, and {L}e {C}am.
\newblock Springer Science \& Business Media, 435.

\end{thebibliography}
